\documentclass[12pt]{article}
\usepackage{amsmath,amssymb,amsthm,mathrsfs,hyperref,theoremref,lipsum,lineno,enumitem,authblk,lmodern,color}

\providecommand{\keywords}[1]{\textbf{\textbf{Keywords}:} #1}
\providecommand{\subjclass}[1]{\textbf{\textbf{Mathematics subject classification(2020)}:} #1}
\usepackage[backend=biber,style=numeric,sorting=nty]{biblatex}
\theoremstyle{plain}
\newtheorem*{thm*}{Theorem}
\newtheorem{thm}{Theorem}[section]

\newtheorem{lem}{Lemma}[section]
\newtheorem{pro}{Proposition}[section]

\theoremstyle{definition}

\newtheorem{rem}{Remark}[section]
\numberwithin{equation}{section}
\allowdisplaybreaks[2]

\newtheorem{lemma}{Lemma}[section]
\newtheorem{proposition}{Proposition}[section]

\newtheorem{conjecture}{Conjecture}[section]
\theoremstyle{definition}
\newtheorem{definition}{Definition}[section]

\theoremstyle{remark}

\def \d {\mathrm{d}}
\def \p {\mathfrak{p}}
\DeclareMathOperator{\tr}{Tr}
\newcommand{\R}{\mathbb{R}}
\newcommand{\Z}{\mathbb{Z}}

\newcommand{\Q}{\mathbb{Q}}
\newcommand{\C}{\mathbb{C}}
\newcommand{\cl}{\mathscr{C}}
\newcommand{\G}{\mathscr{G}}
\newcommand{\Sp}{\mathrm{Sp}}
\newcommand{\SU}{\mathrm{SU}}
\newcommand{\SL}{\mathrm{SL}}
\newcommand{\GSp}{\mathrm{GSp}}
\newcommand{\GL}{\mathrm{GL}}
\newcommand{\Irr}{\mathrm{Irr}}
\newcommand{\ST}{\mathrm{ST}}
\newcommand{\Gal}{\mathrm{Gal}}
\newcommand{\Codim}{\mathrm{Codim}}

\newcommand{\MT}{\mathrm{MT}}
\newcommand{\CST}{\mathrm{CST}}
\newcommand{\AST}{\mathrm{AST}}

\newcommand{\A}{\mathcal{A}}
\newcommand{\one}{\mathbf 1}
\def\dsum{\displaystyle \sum}
\title{Effective Chebotarev-Sato-Tate distribution\\ for abelian varieties with applications}

\author{Mohammed Amin Amri} 
 \newcommand{\Addresses}{{
  \bigskip

  Mohammed Amin ~Amri, \textsc{Department of Mathematics, AGA Laboratory, Ibn Tofail University, Kenitra, Morocco}\par\nopagebreak
  \textit{E-mail address}, Mohammed Amin Amri: \texttt{amri.amine.mohammed@gmail.com}

}}   
\begin{document}

\maketitle
\begin{abstract}
\sloppy In this paper, we aim to establish a uniform effective form of the Chebotarev-Sato–Tate distribution for abelian varieties with a connected Sato-Tate group, conditional on standard analytic conjectures for $L$-functions attached to the relevant irreducible representations. This improves on the exponent of Bucur, Fité, and Kedlaya in \cite[Theorem 1.1]{bucur2024} whenever the Sato–Tate group has a rank greater than one. We also establish a variant for prime-dependent regions of the Sato-Tate group and finite Galois extensions that vary with the counting parameter. As an application, for a normalized non-CM newform $f$ of weight $2$ and level $N$ with a quadratic coefficient field, we obtain a conditional upper bound for $\Pi(x):=\#\{p\le x: p\nmid N,\; a(p)\in\Q\}$, where $a(p)$ is the $p$th Fourier coefficient of $f$. In addition, we formulate a discrepancy conjecture that allows us to prove Murty's asymptotic prediction \cite[Conjecture 3.4]{murty1999} for $\Pi(x)$.
\end{abstract}
 \medskip\\

\section{Introduction}
The estimation of error terms is a central theme in number theory. As Mazur stresses in \cite{mazur2008}, they can encode significant information not apparent in the main term. A telling example is the prime number theorem. Let $\pi(x)$ be the number of primes less than $x$. Consider

$$
 \operatorname{Err}(x):=\pi(x)-\int_2^x \frac{\d t}{\log t}.
$$
The prime number theorem, proved independently by Hadamard and de la Vallée-Poussin, states that $\operatorname{Err}(x)=o(x/\log x)$. Assuming the Riemann Hypothesis for the Riemann zeta function $\zeta(s)$, we obtain a sharper estimate $\operatorname{Err}(x)=O\bigl(x^{1/2}\log x\bigr)$ governed by the nontrivial zeros of $\zeta(s)$.

The same paradigm applies to the distribution of arithmetic functions at prime numbers. An archetypal example is the effective form of the Chebotarev density theorem due to Lagarias and Odlyzko \cite{lagarias1977}, which provides an explicit error term under GRH for the Dedekind zeta function. In the same vein,  Murty \cite{Murty1985} proved a conditional effective form of the Sato--Tate distribution for non-CM elliptic curves over $\Q$, under GRH for symmetric power $L$-functions. In this direction, Bucur, Fité, and Kedlaya \cite{bucur2024} recently extended it to abelian varieties with a connected Sato--Tate group, conditionally on standard analytic conjectures on $L$-functions attached to irreducible characters of the Sato--Tate group and the Mumford--Tate conjecture, building on the general motivic framework outlined in \cite{bucur2016}. Their method uses a multivariate Vinogradov function on a maximal torus to approximate the indicator function of an interval.

Our first purpose is to establish a conditional effective form of the Sato--Tate conjecture for abelian varieties with a connected Sato--Tate group (see \thref{EffST}) using a spectral approximation theorem (see \thref{BSpoly} and \thref{BScor1}) derived from \cite[Theorem 3.1]{colzani2011} by Colzani, Gigante, and Travaglini, which provides spectral majorants and minorants of the characteristic function of an arbitrary measurable subset of the Sato--Tate group with a boundary of finite upper $q$-dimensional Minkowski content. This allows us to derive effective results for a wide class of conjugation-invariant regions of the Sato--Tate group, including prime-dependent ones. For regions of the Sato--Tate group whose boundary has finite codimension one, our result improves on the exponent in Theorem 1.1 of \cite{bucur2024} whenever the Sato––Tate group has rank greater than one. Our approach also improves the exponents in the GRH bounds recorded by Kumar, Kumari, and Mishra \cite[Remark 1.5]{Kumar2026} for joint distributions in planar regions with boundaries that are finite unions of rectifiable curves. 

Our second aim is to establish a quantitative form of the Chebotarev-Sato-Tate conjecture (see \thref{EffCST}) for abelian varieties. This conjecture was formulated by the author in \cite{amri2023} and proved for abelian surfaces potentially of $\GL_2$-type. Such quantitative results serve as tools for addressing  Lang-Trotter-type questions  (see, for instance, \cite{amri2024}). These problems require tracking Frobenius classes in finite Galois groups and normalized Frobenius classes in shrinking regions of the Sato–Tate group, both of which vary with the counting parameters. To handle this simultaneous control, we establish a quantitative form of the Chebotarev-Sato-Tate distribution suited to such situations (see \thref{Unifeffec} and \thref{MovingCST}). Our result also provides a logarithmic improvement over the effective Chebotarev-Sato-Tate estimate in \cite[Theorem 1.2]{wong2019}. 

Building on Van Hirtum's work  \cite{van2017}, we address the quadratic-coefficient case of Murty’s conjecture \cite[Conjecture 3.4]{murty1999}. The quantitative estimate for prime-dependent regions in \thref{MovingCST} yields conditional upper bounds of size $x^{6/7}/(\log x)^{3/7}$ for the number of primes with rational Fourier coefficients. A probabilistic argument suggests the additional cancellation required, leading to a discrepancy conjecture that implies Murty’s predicted asymptotic.

The paper is organized as follows. Section \ref{Sec1} deals with spectral approximation and its formulation in terms of irreducible characters of compact Lie groups. Section \ref{Sec2} sets out the conjectural framework {\it à la} Bucur, Fité, and Kedlaya \cite[\S 2]{bucur2024}. In Section \ref{Sec3}, we prove the effective Sato-Tate and its Chebotarev hybrid version, as well as the prime-dependent version, and apply the results to the quadratic-coefficient case of Murty's conjecture.
\paragraph{Notation and terminology.}
For a compact Riemannian manifold $(M,g)$ of dimension $d$, we write $d_g$ for its Riemannian distance and $\d\mu$ for its normalized Riemannian volume measure, such that $\mu(M)=1$. For a measurable subset $\Omega$ of $M$, we write $\mu(\Omega)$ for its measure with respect to $\d\mu$. For $\Delta>0$, we denote the tube around $\Omega$ by
$$\Omega^{(\Delta)}:=\{x\in M\; :\; d_g(x,\Omega)<\Delta\}.$$
For an integer $0\le q<d$, the upper and lower $q$-dimensional Minkowski contents are
$$
 \overline{\mathcal M}_g^q(\Omega)=
 \limsup_{\Delta\to0^+}\frac{\mu(\Omega^{(\Delta)})}
 {\omega_{d-q}\Delta^{d-q}},\qquad
 \underline{\mathcal M}_g^q(\Omega)=
 \liminf_{\Delta\to0^+}\frac{\mu(\Omega^{(\Delta)})}
 {\omega_{d-q}\Delta^{d-q}},
$$
where $\omega_{d-q}$ is the volume of the unit ball in $\R^{d-q}$. If these quantities agree, we denote their common value by $\mathcal M_g^q(\Omega)$. For a subset $\Omega\subseteq M$ with finite upper $q$-dimensional boundary content $\overline{\mathcal M}_g^q(\partial\Omega)<\infty$, define
\begin{equation}\label{eqboundary}
 C_{\partial\Omega}:=
 \max\left\{1,\sup_{0<\Delta\le1}
    \frac{\mu((\partial\Omega)^{(\Delta)})}{\Delta^{d-q}}\right\}.
\end{equation}
For a family $\boldsymbol\Omega=(\Omega_i)_{i\in I}$ of subsets of $M$, we write $\partial \boldsymbol\Omega$ for the family $(\partial \Omega_i)_{i\in I}$ of the topological boundaries of $\Omega_i$. We define the uniform upper $q$-dimensional Minkowski content of the family $\boldsymbol{\Omega}$ by $$\overline{\mathcal{M}}^q_g(\boldsymbol\Omega):=\limsup_{\Delta\to 0^+}\sup_{i\in I}\frac{\mu(\Omega_i^{(\Delta)})} {\omega_{d-q}\Delta^{d-q}}.$$   
If $\overline{\mathcal{M}}^q_g(\boldsymbol\Omega)<\infty$, we say that the family $\boldsymbol\Omega$ has a finite uniform upper  $q$-dimensional Minkowski content. If $\partial\boldsymbol\Omega$ has a finite uniform upper $q$-dimensional content $\overline{\mathcal{M}}^q_g(\partial\boldsymbol\Omega)<\infty$, we define
\begin{equation}\label{eqboundaryfamily}
 C_{\partial\boldsymbol\Omega}:=
 \max\left\{1,\sup_{i\in I}\sup_{0<\Delta\le1}
 \frac{\mu((\partial\Omega_i)^{(\Delta)})}{\Delta^{d-q}}\right\}.
\end{equation}
Let $f: X \longrightarrow \R$ and $g: X \longrightarrow \R$ be two real-valued functions defined on an ordered set. Throughout the paper, we use Landau's big-$O$ notation $f=O(g)$ to mean that $|f(x)|\le C|g(x)|$ holds for all $x\ge x_0$, for some constant $C>0$, and some $x_0\in X$. We also use the equivalent Vinogradov notation $f\ll g$. We call $C$ the implied constant and write $f\ll_\alpha g$ to emphasize its dependence on a parameter $\alpha$. We write $f\asymp g$ if $f=O(g)$ and $g=O(f)$.

For a number field $F$, a prime $\mathfrak p$ of $F$ refers to a prime ideal of its ring of integers. We denote by $\pi_F(x)$ the number of primes of $F$ with norm at most $x$. For a finite exceptional set $S$ of primes of $F$, we denote by $\pi_F^S(x)$ the number of primes of $F$ outside $S$ with norm at most $x$. When $F=\Q$, we write $\pi(x)$.
\section{Preliminary spectral approximation results}\label{Sec1}
In this section, we derive the essential tools from \cite{colzani2011}  needed for our purposes. We then specialize these tools to compact Lie group settings for use in our application to the effective Chebotarev/Sato-Tate distribution.
\subsection{Spectral approximation for compact Riemannian manifolds}
Let $(M,g)$ be a smooth $d$-dimensional compact, connected Riemannian manifold without boundary, with a Riemannian distance $d_g$ and a normalized Riemannian volume $\mu$ such that $\mu(M)=1$. Let $\Delta_g$ be the non-negative Laplace-Beltrami operator, and let $(\varphi_k)_{k\geq 0}$ be an orthonormal basis of $L^2(M,\d\mu)$ consisting of its eigenfunctions with $\varphi_0=1$. Let $(\lambda_k^2)_{k\ge 0}$ be the corresponding eigenvalues, counted with multiplicity
$$
\Delta_g \varphi_k=\lambda_k^2 \varphi_k.
$$
\\
Let $f\in L^2(M,d\mu)$. For $k\ge 0$, we define the $k$-th Fourier coefficient $\hat{f}(k)$ of $f$ by
$$\hat{f}(k):=\int_M f(x)\overline{\varphi_k(x)}\d\mu(x).$$
With these preliminaries established, we can now formulate the following key result.
\begin{thm}\thlabel{BSpoly}
Let $\Omega\subset M$ be a measurable subset of $M$ such that $\overline{\mathcal{M}}^q_g(\partial \Omega)<\infty$ for some integer $0\le q\le d-1$. Let $C_{\partial \Omega}>0$ be as defined in \eqref{eqboundary}. Fix $\alpha=d+
1$. There exists a constant $c>0$ depending only on $M$, such that for every $T\ge 1$, there exist real-valued spectral polynomials
$$P^\pm_{T}=\sum_{\substack{k\ge 0\\\lambda_k<T}}a_{k,T}^\pm\varphi_k$$
enjoying the following properties.
    \begin{enumerate}[label=(\roman*)]
        \item For every $x\in M$ we have
        $$P^-_{T}(x)\leq \one_\Omega(x)\leq P_{T}^+(x)\quad \text{and} \quad |P_T^+(x)-P_T^-(x)|\leq\frac{c}{(1+T\; d_g(x,\partial \Omega))^{\alpha}}.$$
        \item The constant term of $P^\pm_T$ satisfies
        $$|a^{\pm}_{0,T}-\mu(\Omega)|\ll_{M}\frac{C_{\partial \Omega}}{T^{d-q}}\quad\text{and}\quad |a_{0,T}^{^\pm}|\ll_M 1.$$
        \item We have
        $$\sup_{x\in M}\sum_{0<\lambda_k< T}|a^\pm_{k,T}||\varphi_k(x)|\ll_{M} C_{\partial \Omega}^{1/2}\Psi_q(T)$$
    \end{enumerate}
    where $\Psi_q(T):=T^{q/2}$ for $ q>0$ and $\Psi_0(T):=1+\log T$.
\end{thm}
\begin{proof}
Following the construction in \cite[Theorem 1.1]{colzani2011}, consider a real radial function
$m\in C_c^\infty(\R^d)$, supported in $\{|\xi|<1/2\}$, with
$\int_{\R^d}m(\xi)^2\d\xi=1$, and define the radial multiplier
$$
 h(|\xi|)=(1+|\xi|^2)^{-(d+1)/2}(m*m)(\xi).
$$
Then $h(0)=1$ and $h(t)=0$ for $t\geq1$. Set
\begin{equation}\label{eqkernel}
 K_T(x,y)=\sum_{\lambda_k< T}
 h\left(\dfrac{\lambda_k}{T}\right)\varphi_k(x)\overline{\varphi_k(y)}.
\end{equation}
Let
$$
 H_T(x)=\frac{\beta}{(1+T\,d_g(x,\partial \Omega))^{\alpha}},
$$
where $\beta>0$ is chosen independently of $\Omega$ and $T$ and depends only on $M$, as in the proof of \cite[Theorem 3.1]{colzani2011}. Define
$$
 P_T^\pm(x)=\int_M K_T(x,y)
                  \bigl(\one_\Omega(y)\pm H_T(y)\bigr)\,d\mu(y).
$$
Substituting \eqref{eqkernel} into the displayed equation above and interchanging the integral and the finite sum gives the explicit eigenfunction expansions
$$
P_T^\pm(x)=\sum_{\lambda_k< T}a^\pm_{k,T}
 \varphi_k(x)
$$
where $a^\pm_{k,T}=h\left(\dfrac{\lambda_k}{T}\right)(\widehat{\one}_\Omega(k)\pm \widehat{H}_T(k))$. 
In particular,  since $h(0)=1$ and $\varphi_0=1$, we have
$$
 a_{0,T}^\pm=\mu(\Omega)\pm\int_M H_T\,\d\mu.
$$
Assertion (i) is proved in \cite[Theorem 3.1]{colzani2011}. As for (ii), we estimate $I_\alpha(T):=\int_M H_T\,\d\mu$ as in the proof of \cite[Theorem 3.3]{colzani2011} using the following dyadic partition of $M$
$$M=(\partial \Omega)^{(T^{-1})}\sqcup \bigsqcup_{j\ge 0} D_{j,T}$$
where $D_{j,T}:=\{x\in M \; :\; \frac{2^j}{T}\leq d_g(x,\partial \Omega)<\frac{2^{j+1}}{T}\}$. By our hypothesis on the upper Minkowski content of the boundary of $\Omega$, the constant $ C_{\partial \Omega}$ defined in \eqref{eqboundary} is finite and satisfies
$$
\mu((\partial \Omega)^{(\Delta)})\le C_{\partial \Omega}\; \Delta^{d-q}\quad \text{for every}\quad \Delta>0.
$$
 On $(\partial \Omega)^{(T^{-1})}$, the decay factor $(1+T d_g(x,\partial \Omega))^{-\alpha}$ is at most $1$, whereas on $D_{j,T}$, we have
$$
 (1+T\, d_g(x,\partial \Omega))^{-\alpha}\leq(1+2^j)^{-\alpha}\leq2^{-\alpha j}.
$$
Since $D_{j,T}\subset(\partial \Omega)^{(2^{j+1}/T)}$,  it follows that
\begin{align*}
 I_\alpha(T)
 &=\beta\int_{(\partial \Omega)^{(T^{-1})}}(1+T\,d_g(x,\partial \Omega))^{-\alpha}\,\d\mu
   +\beta \sum_{j=0}^{\infty}\int_{D_{j,T}}(1+Td_g(x,\partial \Omega))^{-\alpha}\,\d\mu\\
 &\leq\beta\mu((\partial \Omega)^{(T^{-1})})+\beta\sum_{j=0}^{\infty}\frac{\mu(D_{j,T})}{2^{\alpha j}}\\
 &\ll_M \dfrac{C_{\partial \Omega}}{T^{(d-q)}}.
\end{align*}
By (i), we have $|P^\pm_T-\one_\Omega|\le P^+_T-P^-_T\ll_M H_T$. Therefore 
\begin{equation}\label{eqnorms}
\|P^{\pm}_T-\one_\Omega\|_{L^1(M)}\ll_M \dfrac{C_{\partial \Omega}}{T^{(d-q)}},\qquad\qquad \|P^{\pm}_T-\one_\Omega\|_{L^2(M)}\ll_M \dfrac{C_{\partial \Omega}^{1/2}}{T^{\frac{d-q}{2}}}
\end{equation}
and 
$$|a_{0,T}^\pm-\mu(\Omega)|\le I_\alpha(T)\ll_M \dfrac{C_{\partial \Omega}}{T^{d-q}}.$$
For the second inequality in (ii), we have $|a_{0, T}^\pm|\le 1+\beta$, since $H_T\le\beta$ and $\mu(M)=1$.
It remains to prove (iii). For $L\geq2$, let $\Pi_L$ be the orthogonal projection onto $E_L:=\operatorname{span}\{\varphi_k\; :\; L\leq\lambda_k<2L\}$. Since $\Pi_L P_{L/2}^+=0$, and $\Pi_L$ is an $L^2$-contraction, by Parseval's identity and \eqref{eqnorms}, one sees
\begin{equation}\label{eqband}
 \left(\sum_{L\leq\lambda_k<2L}|\widehat \one_\Omega(k)|^2\right)^{1/2}
 =\|\Pi_L \one_\Omega\|_{L^2(M)}
 =\|\Pi_L(\one_\Omega-P_{L/2}^+)\|_{L^2(M)}
 \ll_M \frac{C_{\partial \Omega}^{1/2}}{L^{\frac{d-q}{2}}}.
\end{equation}
By \cite[Theorem 5.1, eq (5.3)]{Hormander1968}, we have
\begin{equation}\label{eqweyl}
\sup_{x\in M}\sum_{\lambda_k< L}|\varphi_k(x)|^2\ll_M L^d.   
\end{equation}
For each $x\in M$, by the Cauchy-Schwarz inequality together with \eqref{eqband} and \eqref{eqweyl}, one gets
\begin{align*}
 \sum_{L\leq\lambda_k<2L}|\widehat \one_\Omega(k)|\,|\varphi_k(x)|
 &\leq
 \left(\sum_{L\leq\lambda_k<2L}|\widehat \one_\Omega(k)|^2\right)^{1/2}
 \left(\sum_{\lambda_k<2L}|\varphi_k(x)|^2\right)^{1/2}\\
 &\ll_M C_{\partial \Omega}^{1/2}L^{q/2}.
\end{align*}
We sum over $L=2^j$, $j\geq1$ with $2^j< T$. For $q>0$, the geometric sum contributes $O_q(T^{q/2})$. For $q=0$, each band is $O(C_{\partial \Omega}^{1/2})$, and there are $O(1+\log T)$ bands. Taking into account that the frequencies $\lambda_k$ less than $2$ contribute $O_M(1)$, it follows that 
\begin{equation}\label{eqfbound}
 \sup_{x\in M}\sum_{0<\lambda_k< T}
 |\widehat \one_\Omega(k)|\,|\varphi_k(x)|
 \ll_M C_{\partial \Omega}^{1/2}\Psi_q(T).
\end{equation}
Finally, notice that $a_{k,T}^\pm-\widehat \one_\Omega(k)$ are the $k$th Fourier coefficients of
$P_T^\pm-\one_\Omega$. Therefore, Bessel's inequality, \eqref{eqnorms}, and \eqref{eqweyl}  give the following bound, uniformly in $x$
\begin{align*}
 \sum_{0<\lambda_k< T}
 |a_{k,T}^\pm-\widehat \one_\Omega(k)|\,|\varphi_k(x)|
 &\leq\|P_T^\pm-\one_\Omega\|_{L^2(M)}
       \left(\sum_{\lambda_k< T}|\varphi_k(x)|^2\right)^{1/2}\\
 &\ll C_{\partial \Omega}^{1/2}T^{q/2}.
\end{align*}
Combining this with \eqref{eqfbound} proves (iii).
\end{proof}
We conclude this subsection by recording a uniformity feature of the spectral approximation polynomials from \cite[Theorem 3.1]{colzani2011}. In our arithmetic applications, we shall need spectral polynomials for families of subsets with uniform bounds on all coefficients, which enable us to estimate the error term in question. In the proof of \thref{BSpoly}, the set dependence enters solely through the constant associated with its boundary. This gives the following uniform version of \thref{BSpoly} for families with uniform upper $q$-dimensional Minkowski content.
\begin{pro}\thlabel{BSuniform}
Let $\boldsymbol\Omega:=(\Omega_i)_{i\in I}$ be a family of measurable subsets of $M$ such that $\overline{\mathcal{M}}^q_g(\partial \boldsymbol\Omega)<\infty$ for some integer $0\le q\le d-1$. Let $C_{\partial\boldsymbol \Omega}>0$ be as defined in \eqref{eqboundaryfamily}. Fix $\alpha=d+1$. There exists a constant $c>0$ depending only on $M$, such that for every $T\ge 1$ and $i\in I$ there exist real-valued spectral polynomials
$$P^\pm_{T,\Omega_i}=\sum_{\substack{k\ge 0\\\lambda_k<T}}a_{k,T}^{\pm,i}\varphi_k$$
enjoying the following properties uniformly in $i$.
    \begin{enumerate}[label=(\roman*)]
        \item For every $x\in M$ we have
        $$P^-_{T,\Omega_i}(x)\leq \one_{\Omega_{i}}(x)\leq P_{T,\Omega_i}^+(x)\quad \text{and} \quad | P_{T,\Omega_i}^+(x)-P_{T,\Omega_i}^-(x)|\leq \frac{c}{(1+T\; d_g(x,\partial \Omega_i))^{\alpha}}.$$
        \item The constant term of $P^\pm_{T,\Omega_i}$ satisfies
        $$|a^{\pm,i}_{0,T}-\mu(\Omega_i)|\ll_{M}\frac{C_{\partial \boldsymbol\Omega}}{T^{d-q}},\qquad \text{and}\qquad|a^{\pm,i}_{0,T}|\ll_M 1.$$
        \item We have
        $$\sup_{x\in M}\sum_{0<\lambda_k< T}|a^{\pm,i}_{k,T}||\varphi_k(x)|\ll_{M} C_{\partial \boldsymbol\Omega}^{1/2}\Psi_q(T)$$
    \end{enumerate}
    where $\Psi_q(T):=T^{q/2}$ for $ q>0$ and $\Psi_0(T):=1+\log T$.
\end{pro}
\subsection{Application to Lie groups}\label{liegroups}
In this subsection, we express \thref{BSpoly} in terms of irreducible characters when the Riemannian manifold is a compact Lie group, which will be more useful later for our purposes. We first recall the necessary background from \cite{Bump2013,knapp2002}.\\

Let $G$ be a compact connected reductive Lie group with Lie algebra $\mathfrak{g}$. Fix an $\mathrm{Ad}(G)$-invariant inner product on $\mathfrak{g}$ and equip $G$ with the corresponding bi-invariant Riemannian metric $g$. Then $(G,g)$ is a smooth Riemannian manifold of dimension $d$, whose normalized Riemannian volume measure coincides with its normalized Haar measure $\mu$. Fix a maximal torus $T\subset G$ and let $r$ be its dimension. Set $\mathfrak{t}:=\mathrm{Lie}(T)$ and let $\Phi(G,T)$ be the associated root system. Given a choice of positive roots $\Phi^+$, write 
$$\rho:=\frac{1}{2}\sum_{\alpha\in \Phi^+}\alpha$$
for the Weyl vector and put $s=|\Phi^+|$. Let $\Lambda_G:=X^*(T)$ be the integral weight lattice and let $\Lambda^+_G\subset \Lambda_G$ be the set of dominant weights.
Given $\lambda\in \Lambda_{G}^+$, let $(\pi_\lambda,V_\lambda)$ denote the irreducible unitary representation of $G$ with highest weight $\lambda$, where $V_\lambda$ is equipped with a $G$-invariant Hermitian inner product $\langle\cdot,\cdot\rangle$.
Choose an orthonormal basis $(e_i)_{1\le i\le d_\lambda}$ of $V_\lambda$. Define the normalized matrix coefficient by 
$$m_{i,j}^{\lambda}(g):=\sqrt{d_\lambda}\; \langle e_i,\pi_\lambda(g)e_j \rangle.$$
 The matrix coefficients $(m_{i,j}^{\lambda})_{i,j,\lambda}$ form an orthonormal basis of $L^2(G,\d\mu)$ and satisfy 
$$\Delta_g m_{i,j}^\lambda=c(\lambda)m_{i,j}^\lambda,\qquad c(\lambda):=\|\lambda+\rho\|^2-\|\rho\|^2=(\lambda,\lambda+2\rho),$$
where $(\cdot,\cdot)$ is the inner product on $\mathfrak{t}^*$ induced by the fixed  $\mathrm{Ad}(G)$-invariant inner product on $\mathfrak{g}$ and  $\Delta_g$ is the Laplacian defined by the Riemannian structure on $G$ (cf. \cite[Proposition 5.28(b)]{knapp2002}). Since $G$ is connected, we have $c(\lambda) =0$ if and only if $\lambda=0$.
\begin{pro}\thlabel{BScor1}
 Let $G$ be a compact connected reductive Lie group and let $\Omega$ be a measurable subset of $G$ invariant under conjugation such that $\overline{\mathcal{M}}^q_g(\partial \Omega)<\infty$ for some integer $0\le q\le d-1$. Let $C_{\partial \Omega}>0$ be as defined in \eqref{eqboundary}. For every $T\ge 1$ there are finite linear combinations of irreducible characters of $G$ 
$$Q^\pm_{T}=\sum_{\substack{\lambda\in\Lambda_G^+\\ c(\lambda)<T^2}}c^\pm_{\lambda, T} \chi_\lambda$$
 enjoying the following properties
    \begin{enumerate}[label=(\roman*)]
        \item Pointwise on $G$ we have $$Q^-_{T}\leq \mathbf{1}_\Omega\leq Q_{T}^+.$$
        \item  The constant term satisfies
        $$|c^{\pm}_{0,T}-\mu(\Omega)|\ll_{G}\frac{C_{\partial \Omega}}{T^{d-q}}\quad\text{and}\quad |c^\pm_{0,T}|\ll_G1.$$ 
        \item We have 
        $$\sum_{\substack{\lambda\in\Lambda_G^+\setminus\{0\}\\c(\lambda)< T^2}}|c^\pm_{\lambda,T}|d_\lambda\ll_{G} C_{\partial \Omega}^{1/2}\Psi_q(T).$$
    \end{enumerate}
    where $\Psi_q(T)=T^{q/2}$ for $q>0$ and $\Psi_0(T)=1+\log T$.
\end{pro}
\begin{proof}
For $T\geq 1$, let 
$$P_{T}^\pm=\sum_{\substack{\lambda\in\Lambda_G^+\\c(\lambda)< T^2}}\sum_{1\leq i,j\leq d_\lambda}a_{i,j,T}^{\pm,\lambda}m_{i,j}^\lambda$$ 
be the real-valued spectral polynomials provided by \thref{BSpoly} for $\Omega$ with $d_\lambda:=\dim V_\lambda$. Let 
$$\mathscr{P} : L^2(G)\longrightarrow L^2(G)^G$$
 be the orthogonal projection onto the closed subspace of
conjugation-invariant $L^2(G)$ functions, given by 
$$
\mathscr{P}f(g)=\int_G f(h^{-1}gh)\d\mu(h).
$$
Define 
$$
Q^\pm_{T}:=\mathscr{P} P^\pm_{T}=\sum_{\substack{\lambda\in\Lambda_G^+\\c(\lambda)< T^2}}\sum_{1\leq i,j\leq d_\lambda}a_{i,j,T}^{\pm,\lambda}\mathscr{P}m_{ij}^\lambda.
$$
By the unitarity of $\pi_\lambda$, we have $\pi_\lambda(h^{-1})=\pi_\lambda(h)^*$ . Therefore,
$$m_{i,j}^\lambda(h^{-1}gh)=\sqrt{d_\lambda}\sum_{k,l=1}^{d_\lambda}\overline{\pi_\lambda(h)_{k,i}}\pi_\lambda(g)_{k,l}\pi_\lambda(h)_{l,j}$$
where $\pi_{\lambda}(g)_{i,j}=\langle e_i,\pi_\lambda(g)e_j\rangle$  is the matrix coefficient of $(\pi_\lambda, V_\lambda)$ on $G$. By Schur orthogonality \cite[Theorem 2.4]{Bump2013}, we have
$$
\mathscr{P}m_{i,j}^\lambda(g)=\int_G  m_{i,j}^\lambda(h^{-1}gh)\d\mu(h)=\sqrt{d_\lambda}\sum_{k,l=1}^{d_\lambda}\pi_\lambda(g)_{k,l}\int_G\overline{\pi_\lambda(h)_{k,i}}\pi_\lambda(h)_{l,j} \d\mu(h)=\frac{\delta_{i,j}}{\sqrt{d_\lambda}}\chi_\lambda(g),
$$
where $\chi_\lambda(g)=\displaystyle\sum_{i=1}^{d_\lambda} \pi_\lambda(g)_{i,i}$. Setting $ c^{\pm}_{\lambda,T}:=\frac{1}{\sqrt{d_\lambda}}\displaystyle\sum_{i=1}^{d_\lambda}a^{\pm,\lambda}_{i,i,T}$  gives
$$Q^\pm_{T}=\displaystyle \sum_{\substack{\lambda\in\Lambda_G^+\\ c(\lambda)<T^2}}c^\pm_{\lambda,T}\chi_{\lambda}.$$
Since $\Omega$ is conjugation-invariant, we have $\one_\Omega(h^{-1}gh)=\one_\Omega(g)$ for every $g,h\in G$. Therefore, by applying $\mathscr{P}$ to the inequalities in \thref{BSpoly}(i), one gets (i). As for (ii), it follows immediately from the second point of \thref{BSpoly} as $\mathscr{P}$ preserves integrals. Finally, for (iii), we have
\begin{align*}
 \sum_{0<c(\lambda)< T^2}|c_{\lambda,T}^\pm|d_\lambda
 &\le\sum_{0<c(\lambda)< T^2}\sum_{i,j=1}^{d_\lambda}|a_{i,j,T}^{\pm,\lambda}|\,| m_{i,j}^{\lambda}(e)|
 \nonumber\\
 &\le\sup_{g\in G}\sum_{0<c(\lambda)< T^2}\sum_{i,j=1}^{d_\lambda}|a_{i,j,T}^{\pm,\lambda}|\,|m_{i,j}^\lambda(g)|
 \nonumber\\
 &\ll_G C_{\partial \Omega}^{1/2}\,\Psi_q(T).                   
\end{align*}
Here $e$ denotes the identity element of $G$, and the last inequality follows from \thref{BSpoly}(iii).
\end{proof}
\section{Conjectural framework}\label{Sec2}
In this section, we record the relevant conjectural framework on which the majority of the subsequent results of the paper rely, following the presentation in \cite[\S 2]{bucur2024} and adapting it to the hybrid Chebotarev-Sato-Tate setting by tensoring with Artin representations.\\
Throughout this section, we fix the following notation. Let $K/F$ be a finite Galois extension of degree $n=[K:F]$ with Galois group $G=\Gal(K/F)$. Let $\mathcal{P}(K/F)$ be the set of rational primes $p$ for which some prime of $F$ above $p$ ramifies in $K$. Define
$$
 M(K/F)=2\,nD_F^{1/n_F}\prod_{p\in\mathcal P(K/F)}p,
$$
where $n_F:=[F:\Q]$, and $D_F$ is the absolute discriminant of $F$.
Let $\A$ be an abelian variety of dimension $g$ over $F$ with absolute conductor $N:=N_\A$, that is, the absolute norm of the conductor ideal of $\A$. Set
\begin{equation}\label{eqcond}
  {\cal Q}_{\A,K}=N\cdot M(K/F).    
\end{equation}
  
\subsection{Chebotarev-Sato-Tate group}\label{Sec 3.1}
Following the $\ell$-adic construction in \cite[Chapter 8, \S 8.3]{serre2016}, we shall describe the Lie group data underlying the hybrid Chebotarev--Sato--Tate distribution. For a fixed rational prime $\ell$, let
$$\rho_{\A,\ell} : G_F\longrightarrow\GSp(V_\ell(\A))$$
be the $\ell$-adic representation attached to $\A$ induced by the action of the absolute Galois group $G_F$ of $F$ on the rational $\ell$-adic Tate module $V_\ell(\A)$. Let $G_\ell^{\rm Zar}$ be the Zariski closure of the image of $\rho_{\A,\ell}$ in $\GSp_{2g}(\Q_\ell)$. Put $G_\ell^{1,\rm Zar}:=G_\ell^{\rm Zar}\cap\Sp_{2g,\Q_\ell}$. Fix an isomorphism $\iota_\ell:\overline{\Q}_\ell\cong\C$ and set $G^{\mathrm{Zar}}_{\ell,\C}:=G^{\mathrm{Zar}}_{\ell}\otimes_{\Q_\ell,\imath}\C$ and $G^{1,\mathrm{Zar}}_{\ell,\C}:=G^{1,\mathrm{Zar}}_{\ell}\otimes_{\Q_\ell,\iota}\C$. Fix $F\hookrightarrow\C$, and let $\MT(\A)$ and $\operatorname{Hg}(\A)$  be the Mumford--Tate and Hodge groups associated with the rational Hodge structure on $H_1(\A(\C),\Q)$. In particular, $\operatorname{Hg}(\A)=\MT(\A)\cap\Sp_{2g}$. As in \cite[\S 2.1]{bucur2024}, to avoid a priori dependence of the Sato-Tate group on $\iota_\ell$ and $\ell$, we assume the following conjecture.
\begin{conjecture}[Mumford--Tate conjecture]
\thlabel{conjMT}
For every prime $\ell$,
$$(G_\ell^{\rm Zar})^\circ\cong \MT(\A)\otimes_{\Q}\Q_\ell$$
\end{conjecture}
By \cite{cantoral2022}, the Mumford-Tate conjecture implies the algebraic Sato-Tate conjecture formulated in \cite{BanaszakKedlaya2015}, which we state as follows.
\begin{conjecture}[Algebraic Sato--Tate conjecture]\thlabel{conjAST}
There exists an algebraic subgroup $\AST_F(\A)$ of $\Sp_{2g,\Q}$ such that
$$G_\ell^{1,{\rm Zar}}\cong \AST_F(\A)\otimes_{\Q}\Q_\ell$$
for every prime $\ell$. The group $\AST_F(\A)$ is called the algebraic Sato-Tate group.
\end{conjecture}
This allows us to define the Sato-Tate group independently of $\ell$ as follows.
\begin{definition}
We define the Sato-Tate group of $\A$ to be a maximal compact subgroup $\ST_{F}(\A)$ of $\AST_{F}(\A)(\C)$, and the Chebotarev-Sato-Tate group to be $\CST_K(\A):=\ST_{F}(\A)\times G$.\\
\end{definition}

Let $L/F$ be the finite Galois extension cut out by the component representation
$$
G_F \xrightarrow{\rho_{\A,\ell}} G_\ell^{\rm Zar}
\longrightarrow \pi_0(G_\ell^{\rm Zar}).
$$
This is the connectedness field of $\A$; it is independent of $\ell$ and satisfies
$\Gal(L/F)\simeq \ST_F(\A)/\ST_F(\A)^\circ$.
Let $S$ be the union of the set of primes of bad reduction $S_\A$ of $\A$, those that ramify in $K/F$, and the primes above $\ell$. For $\p\notin S$, we let $\mathrm{Frob}_\p$ denote an arithmetic Frobenius at $\p$ in $G_F$. Let $g_\p$ denote the image of ${\rm Frob}_\p$ under $\rho_{\A,\ell}$. Let $q_{\p}$ be the order of the residue field $\mathbb{F}_{\p}$ at $\p$. Hence, $\tilde{g}_\p:=g_\p q_{\p}^{-1/2}$ is diagonalizable and has eigenvalues of absolute value $1$. Consequently,  it lies in a compact subgroup of $\AST_F(\A)(\C)$. Every compact subgroup of $\AST_F(\A)(\C)$ is conjugate into a maximal compact subgroup. Thus, $\tilde{g}_{\p}$ is conjugate to an element of $\ST_F(\A)$. Define $x_{\p}$ to be the conjugacy class of $\tilde{g}_{\p}\times \sigma_{\p}$ in the Chebotarev-Sato-Tate group $\CST_K(\A)$, where $\sigma_{\p}$ denotes the Frobenius conjugacy class in $G$. The Haar measure $\mu$ on $\CST_{K}(\A)$ is the product of the Haar measure $\mu_{ST}$ of $\ST_F(\A)$ and the normalized counting measure on $G$. Let $X$ denote the set of conjugacy classes of $\CST_K(\A)$. We then have the following hybrid Chebotarev-Sato-Tate conjecture from \cite[Conjecture 2.2]{amri2023}.
\begin{conjecture}[Chebotarev-Sato-Tate Conjecture]\thlabel{conjCST}
Assume that $L$ is linearly disjoint from $K$ over $F$. Then the sequence $(x_{\p})_{\p\notin S}$, with the primes $\p$ ordered by norm, is equidistributed on $X$ with respect to the pushforward of the normalized Haar measure on $\CST_K(\A)$.
\end{conjecture}
\subsection{Motivic $L$-functions}
Let $V$ denote the standard $2g$-dimensional representation of the Sato-Tate group $\ST_F(\A)$ attached to $\A$. Let
$$\sigma:=\varrho\otimes\rho_\xi: \CST_{K}(\A)\longrightarrow \GL(V_{\chi,\xi})$$
be an irreducible representation of $\CST_K(\A)$, where $\chi \in \Irr(\ST_{F}(\A))$ is the irreducible character of the of the $d_{\chi}$-dimensional representation $\varrho:\ST_F(\A)\longrightarrow\GL(V_\chi)$ and $\xi\in \Irr(G)$ is the irreducible character of the $d_\xi$-dimensional Artin representation $\rho_\xi: \Gal(K/F) \to \GL(V_\xi)$, with $V_\chi$ and $V_\xi$ are $\C$-vector spaces. We transport its coefficients via $\iota_\ell^{-1}$ to obtain a $\overline{\Q}_\ell$-representation of $G$, and inflate it to a $d_\xi$-dimensional $\ell$-adic Galois representation 
$$\rho_{\xi,\ell} :G_F\longrightarrow\GL(V_{\xi,\ell}).$$ 
Thus, the family $\left\{\rho_{\xi,\ell}\right\}_\ell$ is a strictly compatible system of weight $0$ and is unramified outside the finite set of primes ramified in $K/F$. 

On the other hand, since $V$ is faithful and self-dual, by virtue of \cite[Proposition 2.20(b)]{DM1982}, $V_\chi$ occurs in some tensor power $V^{\otimes w_\chi}$ for some $w_\chi\in\Z_{\ge0}$. We fix one such $w_\chi$ which we shall call the motivic weight of $\chi$. Via the Weyl unitary trick \cite[\S 5.3.1]{serre2016},  $\varrho$ extends to an algebraic representation $\tilde{\varrho}:G^{1,\rm Zar}_{\ell,\C}\longrightarrow\GL(V_\chi)$. Since $G_{\ell,\mathbb C}^{\mathrm{Zar}}$ contains the scalar
homotheties, multiplication induces a surjective homomorphism
$$
\boldsymbol{m} :
G_{\ell,\C}^{1,\mathrm{Zar}}\times\mathbf G_{m,\C}
\longrightarrow G_{\ell,\C}^{\mathrm{Zar}},
$$
whose kernel is $\langle (-I,-1)\rangle$. The inclusion $V_\chi\subset V^{\otimes w_\chi}$ gives
$\tilde{\varrho}(-I)=(-1)^{w_\chi}\operatorname{Id}_{V_\chi}$.
Consequently, the algebraic representation
$(h,t)\longmapsto t^{w_\chi}\tilde{\varrho}(h)$ is trivial on $\ker \boldsymbol{m}$.
It therefore descends uniquely to an algebraic representation of $G_{\ell,\mathbb C}^{\mathrm{Zar}}$. Transporting its coefficients via $\iota^{-1}_\ell$ yields an algebraic representation $\widetilde\varrho_\ell$ of $G^{\mathrm{Zar}}_\ell$. Define
$$\rho_{\chi,\ell}:=\tilde{\varrho}_\ell\circ\rho_{\A,\ell}\, : G_F\longrightarrow \GL(V_{\chi,\ell}).$$
Therefore, $V_{\chi,\ell}$ is a $d_\chi$-dimensional constituent of $V_\ell(A)^{\otimes w_\chi}$. Hence, $\rho_{\chi,\ell}$ is unramified and pure of weight $w_\chi$ outside $S_{\A}\cup\{\p:\p\mid\ell\}$. Thus, one obtains an $\ell$-adic Galois representation
$$
\rho_{\chi,\xi,\ell}:=\rho_{\chi,\ell} \otimes\rho_{\xi,\ell}: G_F \longrightarrow \GL(V_{\chi,\xi,\ell})
$$
where $V_{\chi,\xi,\ell}:=V_{\chi,\ell}\otimes V_{\xi,\ell}$. For each prime $\p$ of $F$ outside $S$, it satisfies
$$
\det\Bigl(1 - \rho_{\chi,\xi,\ell}({\rm Frob}_{\p})\,q_\p^{-w_{\chi}/2} T\Bigr) = \det\bigl(1 - \sigma(x_{\p})\,T\bigr).
$$
For a prime $\p$ of $F$, define the local Euler polynomial of $\chi\otimes\xi$ at $\p$ by
$$
L_{\p}(\chi\otimes\xi,T) :=\det\Bigl(1 - \rho_{\chi,\xi,\ell}({\rm Frob}_{\p})\,q_\p^{-w_{\chi}/2} T \;\Big|\; (V_{\chi,\xi,\ell})_{I_{\p}}\Bigr),
$$
where $I_{\p}$ denotes the inertia subgroup of the decomposition group $G_{\p}$ at ${\p}$ inside the absolute Galois group $G_F$ and $(V_{\chi,\xi,\ell})_{I_{\p}}$ denotes the inertia coinvariants. The polynomials $L_{\p}(\chi\otimes\xi,T)$ are independent of $\ell$ for $\p\notin S$, their degree is at most $d_{\chi,\xi}=d_\chi d_\xi$, and all their reciprocal roots have absolute value at most $1$, with equality for $\p$ outside $S$.  Define the motivic $L$-function attached to $\chi\otimes\xi$ as
$$
L(\chi\otimes\xi,s):= \prod_{\p} L_{\p}\bigl(\chi\otimes\xi,q_\p^{-s}\bigr)^{-1},
$$
which converges absolutely for $\Re(s)>1$. Let 
$$
B_{\chi,\xi}:=D_F^{d_{\chi,\xi}} N_{\chi,\xi},
$$
where $N_{\chi,\xi}:=\displaystyle \prod_{\p}N\p^{f_{\chi,\xi}(\p)}$ is the norm of the conductor of $\rho_{\chi,\xi,\ell}$. Here, $f_{\chi,\xi}(\p)$ is the local conductor exponent; see \cite[\S 2]{serre1970}.  
Let $\kappa_{\chi,\xi,j}\in \Q$ be the parameters of the normalized archimedean factors, constructed as in \cite[\S 3]{serre1970}, satisfying $0\leq\kappa_{\chi,\xi,j}\leq1+w_\chi/2$  for $j=1,\dots,n_F\,d_{\chi,\xi}$. Define the completed $L$-function as
$$
\Lambda(\chi\otimes\xi,s) := B_{\chi,\xi}^{s/2}\, L(\chi\otimes\xi,s)\, \Gamma(\chi\otimes\xi,s)\quad\text{where}\quad\Gamma(\chi\otimes\xi,s) :=\prod_{j=1}^{n_{\chi,\xi}} \Gamma_{\R}(s + \kappa_{\chi,\xi,j}),
$$
with $\Gamma_\R(s)=\pi^{-s/2}\Gamma(s/2)$ and $n_{\chi,\xi}:=n_F d_{\chi,\xi}$. Put
$$
\delta_\chi:=\dim_{\mathbb C}V_\chi^{\ST_F(\A)},
\qquad
\delta_\xi:=\dim_{\mathbb C}V_\xi^{G}.
$$
Then
$$
\delta_{\chi,\xi}:=\dim_{\mathbb C}(V_\chi\otimes V_\xi)^{\ST_F(\A)\times G}=\delta_\chi\delta_\xi.
$$
Since $\chi$ and $\xi$ are irreducible, $\delta_{\chi,\xi}=1$ precisely when both characters are trivial, and $\delta_{\chi,\xi}=0$ otherwise.
\begin{conjecture}\thlabel{conjGRH}
Suppose that $L$ is linearly disjoint
from $K$ over $F$.  For an irreducible character $\chi\otimes\xi$  of $\CST_{K}(\A)$, the following hold.
\begin{enumerate}[label=(\roman*)]
\item The function $s^{\delta_{\chi,\xi}}(s-1)^{\delta_{\chi,\xi}}\,\Lambda(\chi\otimes\xi,s)$ extends to an entire function of order 1 that does not vanish at $s=0,1$.
\item There exists $\varepsilon_{\chi\otimes\xi} \in \C$ with $|\varepsilon_{\chi\otimes\xi}|=1$ such that 
$$
\Lambda(\chi\otimes\xi,s) = \varepsilon_{\chi\otimes\xi}\, \Lambda((\chi\otimes\xi)^{\vee},1-s),
$$
for all $s \in \C$, where $(\chi\otimes\xi)^{\vee}$ is the character of the contragredient representation of $\chi\otimes\xi$.
\item Every zero of the completed $L$-function has real part $1/2$.
\end{enumerate}
\end{conjecture} 
The following hybrid version of \cite[Proposition 2.5]{bucur2024} will play a central role in the next section. 
\begin{pro}\thlabel{Murtyestim}
Assume that \thref{conjGRH} holds for the irreducible character $\chi\otimes\xi$ of $\CST_K(\A)$. For every $x\ge 2$, we have
$$
\sum_{\substack{N\p\le x\\\p\notin S}}\chi\otimes\xi(x_\p)\,\log N\p=\delta_{\chi,\xi}\,x+O_{F,g}\left(d_{\chi,\xi}\,x^{1/2}(\log x)\log\left({\cal Q}_{\A,K}(x+w_{\chi})\right)\right).
$$
Abel summation then gives 
$$
\sum_{\substack{N\p\le x\\
\p\notin S}}
\chi\otimes\xi(x_\p)
=\delta_{\chi,\xi}\,\mathrm{Li}\,x+O_{F,g}\left(d_{\chi,\xi}\,x^{1/2}\log\left({\cal Q}_{\A,K}(x+w_{\chi})\right)
\right).
$$
\end{pro}
 The proof of the above proposition follows the same pattern as in \cite[Proposition 4.1]{Murty1985}. It involves the following conductor estimate.
\begin{lem}\thlabel{lemconductor} We have
$$\log B_{\chi,\xi} \ll_{F,g} d_{\chi,\xi}\log{\cal Q}_{\A,K}.$$
\end{lem}
\begin{proof}
Let  $B_\xi=D_F^{d_\xi}N_\xi$ and $B_\chi:=D_F^{d_\chi}N_\chi$, where $N_\chi$ (resp. $N_\xi$) is the norm of the conductor of $\rho_{\chi,\ell}$ (resp. $\rho_{\xi,\ell}$).
 By Hensel's estimate \cite[\S 2.4, (2.5)–(2.7)]{wong2019}, one has
\begin{equation}\label{eqhensel}
 \log B_\xi
       \ll n_Fd_\xi\log M(K/F),
\end{equation}
with an absolute implied constant. On the other hand, by \cite[Remark 2.8]{bucur2024}
\begin{equation}\label{eqbfkcond}
    \log B_\chi\ll_{F,g} d_\chi\log 2N.
\end{equation}
For $u\ge -1$, let $G^u$ denote the upper ramification groups of the local absolute Galois group $G_{F_\p}$.
The inclusion
$V_{\chi,\ell}^{G^u}\otimes V_{\xi,\ell}^{G^u}\subseteq(V_{\chi,\ell}\otimes V_{\xi,\ell})^{G^u}$ implies
$$
\Codim(V_{\chi,\ell}\otimes V_{\xi,\ell})^{G^u}
\leq d_\xi \,\Codim V_{\chi,\ell}^{G^u}+d_\chi\, \Codim V_{\xi,\ell}^{G^u}.     
$$
Integrating this inequality in the upper-numbering formula for the Artin
conductor gives 
$$
 f_{\chi,\xi}(\p)\leq d_\xi f_\chi(\p)+d_\chi f_\xi(\p).
$$
Therefore,  $ N_{\chi,\xi}\leq N_\chi^{d_\xi}N_\xi^{d_\chi}$. Thus,
$$ B_{\chi,\xi}=D_F^{d_{\chi,\xi}}N_{\chi,\xi}
 \leq B_\chi^{d_\xi}N_\xi^{d_\chi}
  =\frac{B_\chi^{d_\xi}B_\xi^{d_\chi}} {D_F^{d_{\chi,\xi}}}.$$
Combining \eqref{eqhensel} and \eqref{eqbfkcond} gives the desired bound.
\end{proof}
\paragraph{Convention.} Throughout the rest of the paper, we adopt the convention that all sums involving the classes $x_\p$ or $\tilde{g}_\p$ are taken over primes $\p$ outside $S$ and $S_\A\cup\{\p: \p\mid\ell\}$, respectively.\\

Next, we record an analogue of the motivic weight estimate \cite[Proposition 3.2(iii)] {bucur2024} in the form needed for our purposes.
\begin{pro}\thlabel{wieghtbound}
Assume that $\ST_F(\A)$ is connected. Fix a bi-invariant Riemannian metric and a maximal torus $T$, and let $\mathfrak{t}$ be its Lie algebra. Then, the $\ell$-adic representations $\rho_{\chi_\lambda,\ell}$ constructed above for the irreducible characters of $\ST_F(\A)$ of highest weights $\lambda$ can be chosen with purity weights satisfying
    $$w_\lambda\le \varsigma\,\|\lambda\|$$
    where $\|\cdot\|$ is the dual norm on $\mathfrak{t}^*$ induced by the metric. The constant $\varsigma$ is independent of $\lambda$ and $\ell$
and depends on $\ST_F(\A)$, the metric, the standard representation,
and fixed finite tensor choices.
\end{pro}
\begin{proof}
 By Tannakian duality and complete reducibility \cite[Proposition 2.20(b)]{DM1982}, every irreducible representation of $\ST_F(\A)$ with highest weight $\lambda$ occurs as a constituent of some tensor power $V^{\otimes^{w_\lambda}}$. We shall choose these tensor powers quantitatively. 
 
 Fix a root system $\Phi:=\Phi(\ST_F(\A),T)$ of $\ST_F(\A)$ with respect to $T$ of rank $r$. Choose simple roots $\alpha_1,\dots,\alpha_r\in\Phi$. Let $\Lambda:=X^{*}(T)\subset\mathfrak{t}^*$ be the integral weight lattice, and let  
 $$\Lambda_0:=\{\mu\in\Lambda\; :\; \langle \mu,\alpha_{i}^\vee\rangle=0\; \text{for all}\; i\},$$
 where $\langle\cdot,\cdot\rangle :\Lambda\times\Lambda^\vee\longrightarrow\Z$ is the natural pairing.
 Let $\mathcal{V}_{ss}$ be the $\R$-span of $\Phi$. Then, by \cite[\S 7.4,  7.4.2, p 124]{springer1998}, we have
 $$\mathfrak{t}^{*}=\mathcal{V}_{ss}\oplus(\Lambda_0\otimes\R).$$
Choose $(\omega_i)_{1\le i\le r_{ss}}$ fundamental weights in ${\cal V}_{ss}$, and a lattice basis $(\eta_j)_{1\le j\le z}$ of $\Lambda_0$. Since the fundamental weights are rational linear combinations of the roots, there exists an integer $e\ge 1$ such that $e\omega_i\in\Lambda$ for all $1\le i\le r_{ss}$. It then follows that
$$
\mathfrak{t}^*=\bigoplus_{i=1}^{r_{ss}}\R e\omega_i\oplus\bigoplus_{j=1}^z\R\eta_j.
$$
For $\mu:=\dsum_i u_i(e\omega_i)+\dsum_j v_j\eta_j\in \mathfrak{t}^*$, define the coordinate norm
$$\|\mu\|_1:=\dsum_i |u_i|+\dsum_j|v_j|.$$ 
By equivalence of norms, we have $\|\mu\|_1\le C_0\|\mu\|$ for some $C_0>0$. For a dominant weight $\lambda$ of $\ST_F(\A)$ write
$$
\lambda=\sum_i c_i(e\omega_i)+\sum_j t_j\eta_j,
$$
with $c_i=e^{-1}\langle \lambda,\alpha_i^{\vee}\rangle\ge 0$. Set $a_i=\lfloor c_i\rfloor$, $b_j=\operatorname{sgn}(t_j)\lfloor|t_j|\rfloor$,
and put
$$
\nu_\lambda
=\lambda-\sum_i a_i(e\omega_i)-\sum_j b_j\eta_j.
$$
Then $\nu_\lambda\in\Lambda$ and $\langle \nu_\lambda,\alpha_i^\vee\rangle=e(c_i-a_i)\ge 0$. Moreover, its coordinates satisfy $0\le c_i-a_i< 1$ and $|t_j-b_j|<1$. Therefore, the $\nu_\lambda$'s form a finite set $\mathcal R$ of dominant weights. Set 
$$\mathcal{F}:=\mathcal{R}\cup\{e\omega_i\; :\; 1\le i\le r_{ss}\}\cup\{\pm\eta_j\; :\; 1\le j\le z\}.$$
Note that every weight in $\mathcal{F}$ is integral and dominant. For each $\gamma\in\mathcal F$, fix an occurrence of $V_\gamma$ as a constituent of some tensor power $V^{\otimes n_\gamma}$ for some $n_\gamma\in\mathbb Z_{\ge0}$
 (with $n_0=0$).  By highest weight theory, $V_\lambda$ is the Cartan constituent of
$$
V_{\nu_\lambda}\otimes
\bigotimes_i V_{e\omega_i}^{\otimes a_i}\otimes
\bigotimes_{b_j\ne0}
V_{\operatorname{sgn}(b_j)\eta_j}^{\otimes |b_j|}.  
$$
Consequently, $V_\lambda$ occurs in $V^{\otimes w_\lambda}$ with
$$
 w_\lambda=n_{\nu_\lambda}+\sum_i a_i n_{e\omega_i}+\sum_{\substack{j\\ b_j\ne0}}|b_j| n_{\operatorname{sgn}(b_j)\eta_j}.
$$
Consequently, we have
$$w_\lambda\le C\left(1+\|\lambda\|_1\right)\quad\text{with}\quad C=\displaystyle\max_{\gamma\in\mathcal F}n_\gamma.$$
Since $\Lambda$ is a lattice, there exists $\delta>0$ such that every nonzero $\lambda\in\Lambda$ satisfies $\|\lambda\|\ge\delta$. Hence, $$w_\lambda\le C(\delta^{-1}+C_0)\|\lambda\|.$$
\end{proof}
\section{Arithmetic applications}\label{Sec3}
In this section, we derive effective Sato-Tate and Chebotarev-Sato-Tate estimates using the real-valued spectral polynomials from Section \ref{Sec1} under the conjectural framework described in Section \ref{Sec2}. We also address the case of Murty’s conjecture with a quadratic coefficient field.
\subsection{Effective Sato-Tate distribution}
We recall the notation from Subsection \ref{liegroups}, which we use throughout this section. Let $\A$ be an abelian variety over a number field $F$ with a connected Sato-Tate group $\ST_F(\A)$.  Fix a bi-invariant Riemannian metric on $\ST_F(\A)$  and write $\mu_{\ST}$ for its normalized Haar measure. Let $T$ be a maximal torus of $\ST_F(\A)$, and let  $r$ denote its dimension, which is the rank of $\ST_F(\A)$. Let $\Phi(\ST_F(\A),T)$ be the associated root system. Choose a system of positive roots $\Phi^+$, and let $s$ denote its number $|\Phi^+|$. For $0\le q\le d-1$, put
\begin{equation}\label{eqexpoST}
    \kappa:=\dfrac{2(d-q)}{2d-q}.
\end{equation}
Write $\Lambda$ for the weight lattice, and $\Lambda^+\subset\Lambda$ for the subset of dominant weights. Let
$$
\rho=\dfrac{1}{2}\sum_{\alpha\in \Phi^+}\alpha
$$
be the Weyl vector. Let $c(\lambda):=\|\lambda+\rho\|^2-\|\rho\|^2$ be the Casimir eigenvalue attached to $\lambda$. For $\lambda\in\Lambda^+$, we have
\begin{equation}\label{eqcas}
    \|\lambda\|^2\le c(\lambda).
\end{equation}
We can now state the following strengthening of \cite[Theorem 1.1]{bucur2024}.
\begin{thm}\thlabel{EffST}
Let $F$ be a number field, and let $\A/F$ be an abelian variety with a connected Sato–Tate group $\ST_F(\A)$. Let $\Omega\subset \ST_F(\A)$ be a measurable conjugation-invariant subset satisfying $\overline{\mathcal{M}}^q_g(\partial \Omega)<\infty$, with $0< q\le d-1$ and $d=\dim\ST_F(\A)$. Assume that \thref{conjMT} holds for $\A$, and \thref{conjGRH} holds for every irreducible character of $\ST_F(\A)$. Then, for $x\ge 3$
$$\#\{N\p\le x \; :\; \; \p\notin S_\A, \;\;\tilde{g}_\p\in \Omega\}=\mu_{\ST}(\Omega)\pi_F(x)+O\left( C_{\partial \Omega}\dfrac{x^{1-\frac{\kappa}{2}}(\log N x)^{\kappa}}{(\log x)^{1-\kappa}}\right)$$
 where $N$ is the absolute conductor of $\A$, $C_{\partial \Omega}$ is defined in \eqref{eqboundary}, and $\kappa$ is defined in \eqref{eqexpoST}. The
implied constant depends only on the geometry of $\ST_F(\A)$, the field $F$, and the abelian variety $\A$.
\end{thm}
\begin{proof}
 Write 
 \begin{equation}\label{eqsum}
  \#\{N\p\le x \; :\; \; \p\notin S_\A, \;\;\tilde{g}_\p\in \Omega\}=\sum_{N\p\le x}\one_\Omega(\tilde{g}_\p).   
 \end{equation}
 For $T\ge1$, \thref{BScor1} supplies the real-valued spectral polynomials
$$
 Q_T^-\le\one_\Omega\le Q_T^+,\qquad Q_T^\pm(g)=\sum_{c(\lambda)< T^2} c_{\lambda,T}^\pm\chi_\lambda(g).
$$
We bound \eqref{eqsum} from above
\begin{eqnarray}
\sum_{N\p\le x}\one_\Omega(\tilde{g}_\p)&\le&
 \sum_{N\p\le x}Q_T^+(\tilde{g}_\p)\nonumber\\
& \le & c_{0,T}^+\, \pi_F^{S_\A}(x) +\sum_{0<c(\lambda)< T^2} c_{\lambda,T}^+ \sum_{N\p\le x}\chi_\lambda(\tilde{g}_\p)\nonumber\\
&\le & \mu_{\ST}(\Omega)\pi^{S_\A}(x) +\underbrace{(c_{0,T}^+-\mu_{\ST}(\Omega))\, \pi_F^{S_\A}(x)}_{E_1^+}+\underbrace{\sum_{0<c(\lambda)< T^2} c_{\lambda,T}^+\sum_{N\p\le x}\chi_\lambda(\tilde{g}_\p)}_{E_2^+}.\nonumber
\end{eqnarray}
Likewise, we bound \eqref{eqsum} from below
\begin{eqnarray}
\sum_{N\p\le x}\one_\Omega(\tilde{g}_\p)&\ge&
 \sum_{N\p\le x}Q_T^-(\tilde{g}_\p)\nonumber\\
& \ge & c_{0,T}^-\, \pi_F^{S_\A}(x) +\sum_{0<c(\lambda)< T^2} c_{\lambda,T}^- \sum_{N\p\le x}\chi_\lambda(\tilde{g}_\p)\nonumber\\
&\ge & \mu_{\ST}(\Omega)\pi_F^{S_\A}(x) +\underbrace{(c_{0,T}^- -\mu_{\ST}(\Omega))\, \pi_F^{S_\A}(x)}_{E_1^-}+\underbrace{\sum_{0<c(\lambda)< T^2} c_{\lambda,T}^-\sum_{N\p\le x}\chi_\lambda(\tilde{g}_\p)}_{E_2^-}.\nonumber
\end{eqnarray}
It follows that
\begin{equation}\label{eq1}
\left|\dsum_{N\p\le x}\one_\Omega(\tilde{g}_\p)-\mu_{\ST}(\Omega)\pi_F^{S_\A}(x)\right|\le\max\{|E_1^\pm|+|E_2^\pm|\}.   
\end{equation}
\thref{BScor1}(ii) yields
$$|c_{0,T}^\pm-\mu_{\ST}(\Omega)|\ll_{\ST_F(\A)} \frac{C_{\partial \Omega}}{T^{d-q}}.$$
The Chebyshev upper bound gives $\pi^{S_\A}_F(x)\ll_F\frac{x}{\log x}$. Consequently,
\begin{equation}\label{eq2}
   |E_1^\pm|\ll_{\ST_F(\A),F} \frac{C_{\partial \Omega}}{T^{d-q}}\frac{x}{\log x}. 
\end{equation}
On the other hand, one has 
\begin{equation}\label{eq3}
    |E_2^\pm|\le \sum_{0<c(\lambda)< T^2} |c_{\lambda,T}^\pm|\left|\sum_{N\p\le x}\chi_\lambda(\tilde{g}_\p)\right|.
\end{equation}
\thref{Murtyestim} shows that
$$
\sum_{N\p\le x}\;\chi_\lambda(\tilde{g}_{\p})=O_F(d_\lambda\;x^{1/2}\log(N(x+w_\lambda))).
$$
By \thref{BScor1}(iii), we have
\begin{equation}\label{eq4}
    \sum_{0<c(\lambda)< T^2} |c_{\lambda,T}^\pm|d_\lambda\ll_{\ST_F(\A)} C_{\partial \Omega}^{1/2}T^{q/2}.
    \end{equation}   
Let $\varsigma>0$ be the constant from \thref{wieghtbound}. Then by \eqref{eqcas}, we have $w_\lambda\le \varsigma\, T$. Thus \eqref{eq3} becomes
\begin{equation}\label{eq5}
  |E_2^\pm|\ll_{\ST_F(\A)}  C_{\partial \Omega}^{1/2}T^{q/2} x^{1/2}\log(N(x+\varsigma T)).
\end{equation}
Since $C_{\partial \Omega}\ge 1$, we may replace its square root by $C_{\partial\Omega}$ in \eqref{eq5}. Consequently,
\begin{equation*}
\left|\sum_{N\p\le x}\one_\Omega(\tilde{g}_\p)-\mu_{\ST}(\Omega)\;\pi^{S_\A}_F(x) \right|\ll_{\ST_F(\A)} \frac{C_{\partial \Omega}}{T^{d-q}}\frac{x}{\log x}+C_{\partial \Omega}T^{q/2}x^{1/2}\log(N(x+\varsigma T)).
\end{equation*}
To balance the two  error terms, one may take
$$
 T=\left(\frac{x^{1/2}}{\log(Nx)\log x}\right)^{\frac{2}{2d-q}}.
$$
This is at least $1$ for sufficiently large $x$ and is bounded by a fixed power of $x$. In particular, we have $\log (N(x+\varsigma T))\asymp_{\ST_F(\A)}\log Nx$. Therefore, \eqref{eq1} yields 
$$
\#\{N\p\le x \; :\; \; \p\notin S_\A, \;\;\tilde{g}_\p\in \Omega\}=\mu_{\ST}(\Omega)\;\pi_F^{S_\A}(x)+O\left(C_{\partial \Omega}\dfrac{x^{1-\frac{\kappa}{2}}(\log Nx)^{\kappa}}{(\log x)^{1-\kappa}}\right).
$$
Finally, notice that 
$$|\pi_F(x)-\pi_F^{S_\A}(x)|\leq\#S_{\A}\ll_F\log N.$$
Replacing $\pi_F^{S_\A}(x)$ with $\pi_F(x)$ therefore costs $O_F(\log N)$, which can be absorbed into the error term.
\end{proof}
\begin{rem}
\begin{enumerate}
\item In the case when $q=0$, one may similarly derive 
$$
\#\{N\p\le x \; :\; \; \p\notin S_\A, \;\;\tilde{g}_\p\in \Omega\}=\mu_{\ST}(\Omega)\;\pi_F(x)+O\left(C_{\partial \Omega}x^{\frac{1}{2}}(\log Nx)(\log x)\right).
$$
\item  Let $I\subset [-2g,2g]$ be an interval, where $g$ is the dimension of the abelian variety $\A$. By Hironaka’s uniformization theorem \cite[Theorem 0.1]{bm1988}, the boundary of the inverse image of $I$ under the trace map has finite upper $(d-1)$-dimensional Minkowski content. Therefore, the power of $x$ in
\thref{EffST} is $1-1/(d+1)$. If $r$ is the rank and
$s$ the number of positive roots, then $d=r+2s$.
The exponent in \cite[Theorem 3.8]{bucur2024} is
$1-1/(2(r+s))$. Thus, our exponent is smaller when $r>1$, and the exponents agree when $r=1$. 
\end{enumerate}
\end{rem}
\subsection{Effective Chebotarev-Sato-Tate distribution}
We retain the notation from the preceding subsection, recall the character expansion of a class function, and introduce notation for a conjugation-invariant subset of a finite group. Let $K/F$ be a finite Galois extension with Galois group $\G=\Gal(K/F)$, and let
$\cl\ne\varnothing$ be a conjugation-invariant subset of $\G$. For $\xi\in\Irr(\G)$, write
$$
b_\xi(\cl)=\frac1{|\G|}\sum_{\gamma\in\cl}\overline{\xi(\gamma)},
 \qquad
 \mathscr A_\G(\cl)=\sum_{\xi\in\Irr(\G)}d_\xi|b_\xi(\cl)|.
$$
By orthogonality and the regular-representation dimension formula,  we have
$$
\sum_\xi|b_\xi(\cl)|^2=\frac{|\cl|}{|\G|},
\qquad
\sum_\xi d_\xi^2=|\G|.
$$
Furthermore, the Cauchy-Schwarz inequality yields
$$
\mathscr A_{\G}(\cl)
\le
\left(\sum_\xi d_\xi^2\right)^{1/2}
\left(\sum_\xi|b_\xi(\cl)|^2\right)^{1/2}
=|\cl|^{1/2}.
$$
We have the following effective form of \thref{conjCST}.
\begin{thm}\thlabel{EffCST}
Let $F$ be a number field, let $K/F$ be any finite Galois extension with Galois group $\G$, and let $\varnothing\ne\cl\subseteq\G$ be conjugation-invariant. Let $\A$ be an abelian variety with a connected Sato–Tate group $\ST_F(\A)$. Let $\Omega\subset\ST_F(\A)$ be a measurable conjugation-invariant subset satisfying $\overline{\mathcal{M}}^q_g(\partial \Omega)<\infty$ where $0< q\le d-1$ with $d=\dim\ST_F(\A)$. Assume that \thref{conjMT} holds for $\A$ and \thref{conjGRH} holds for every irreducible character of $\CST_K(\A)$. If
$$
 \frac{|\cl|^{1/2}\,x^{1/2}}
      {|\G|(\log{\cal Q}_{\A,K}x)\log x}\geq1,
$$
then
$$
 \#\{N\mathfrak p\leq x\,:\,
     \widetilde g_{\mathfrak p}\in \Omega,\ \sigma_{\p}\in \cl\}\notag
 =\frac{|\cl|}{|\G|}\mu_{\ST}(\Omega)\pi_F(x)
 +O\left(C_{\partial \Omega}
   \frac{|\cl|^{1-\frac{\kappa}{2}}}{|\G|^{1-\kappa}}
   \frac{x^{1-\kappa/2}(\log {\cal Q}_{\A,K}x)^{\kappa}}{(\log x)^{1-\kappa}}
   \right)
$$
 where ${\cal Q}_{\A,K}$ is defined in \eqref{eqcond},  $C_{\partial \Omega}$ is defined in \eqref{eqboundary}, and $\kappa$ is defined in \eqref{eqexpoST}. The implied constant in the above error term depends only on the geometry of $\ST_F(\A)$, the field $F$, and the abelian variety $\A$.
\end{thm}
\begin{proof}
Write 
$$
\#\{N\p\le x \; :\; \tilde{g}_\p\in \Omega,\quad \sigma_{\p}\in \cl\}= \sum_{N\p\le x}\one_\Omega(\tilde{g}_{\p})\mathbf{1}_\cl(\sigma_\p).$$
Let $T\ge 1$ and let $Q^\pm_T$  be the real-valued spectral polynomials supplied by \thref{BScor1}. Since $\one_\cl\ge0$, we have
\begin{equation}\label{eq:1}
 Q_T^-(\tilde{g}_\p)\one_\cl(\sigma_\p)
 \le\one_\Omega(\tilde{g}_\p)\one_\cl(\sigma_\p)
 \le Q_T^+(\tilde{g}_\p)\one_\cl(\sigma_\p).
\end{equation}
 Since the indicator function $\mathbf{1}_\cl$ is constant on conjugacy classes of $\G$, it is a class function. Hence, 
$$
 \one_\cl=\displaystyle\sum_{\xi\in\Irr(\G)}b_{\xi}(\cl)\xi
 $$
 because irreducible characters of $\G$ form an orthonormal basis of the space of class functions.
Expanding the two real-valued bounding spectral polynomials gives
\begin{eqnarray}\label{eq:2}
 \sum_{N\p\le x}
   Q_T^\pm(\widetilde g_\p)\one_\cl(\sigma_\p)&=&\sum_{\substack{c(\lambda)<T^2\\\xi\in\Irr(\G)}}
 c_{\lambda,T}^\pm \; b_{\xi}(\cl)
 \sum_{N\p\le x}
      \chi_\lambda(\widetilde g_\p)\xi(\sigma_\p)\nonumber\\
  &=&   \frac{|\cl|}{|\G|}\mu_{\ST}(\Omega)\pi_F^{S}(x)+\underbrace{\frac{|\cl|}{|\G|} (c_{0,T}^\pm-\mu_{\ST}(\Omega))\pi_F^S(x)}_{E_1^\pm}\nonumber\\
  &+&
\underbrace{c_{0,T}^{\pm}\sum_{1\neq\xi\in \Irr(\G)}b_{\xi}(\cl)\sum_{N\mathfrak{p}\le x}\xi(\sigma_\mathfrak{p})}_{E_2^\pm}\nonumber\\
&+&\underbrace{\sum_{0<c(\lambda)<T^2}
c_{\lambda,T}^{\pm}
\sum_{\xi\in\operatorname{Irr}(\G)}b_{\xi}(\cl)\sum_{N\p\le x}
      \chi_\lambda(\widetilde g_\p)\xi(\sigma_\p)}_{E^\pm_3}.
\end{eqnarray}
Summing \eqref{eq:1} over primes $\p$ outside $S$ with norm at most $x$ and invoking \eqref{eq:2} and the triangle inequality yields
\begin{equation}\label{eq:3}
\left|\sum_{N\p\le x}\one_\Omega(\tilde{g}_\p)\one_\cl(\sigma_\p)-\frac{|\cl|}{|\G|}\mu_{\ST}(\Omega)\pi_F^{S}(x)\right|\le \max\{|E_1^\pm|+|E^\pm_2|+|E_3^\pm|\}.
\end{equation}
By \thref{BScor1}(ii) and $\pi_F(x)\ll_F x/\log x$, one has
\begin{equation}\label{eq:4}
    |E_1^\pm|\ll_{\ST_F(\A)} \frac{|\cl|}{|\G|} \frac{x\, }{\log x}\frac{C_{\partial \Omega}}{T^{d-q}}.
\end{equation}
\thref{wieghtbound} together with \eqref{eqcas} implies that all the weights in the sums $E^\pm_2$ and $E^\pm_3$ satisfy $w_\lambda\le\varsigma T$. By \thref{BScor1}(ii), we have $|c_{0,T}^\pm|=O_{\ST_F(\A)}(1)$.  Therefore, it follows from \thref{Murtyestim}  that
\begin{eqnarray}\label{eq:5}
|E_2^\pm|&\ll_{\ST_F(\A)}& \left(\sum_{1\neq\xi\in\Irr(\G)}|b_{\xi}(\cl)|d_\xi\right)  x^{1/2}\log({\cal Q}_{\A,K}(x+\varsigma T))\nonumber\\
&\ll_{\ST_F(\A)}&\left(\mathscr{A}_\G(\cl)-\frac{|\cl|}{|\G|}\right) x^{1/2}\log({\cal Q}_{\A,K}(x+\varsigma T)),  
\end{eqnarray}
and
\begin{eqnarray}\label{eq:6}
 |E^\pm_3|&\ll_{\ST_F(\A)}&
 \left(\sum_{0<c(\lambda)<T^2}|c_{\lambda,T}^\pm|d_\lambda\right)
 \left(\sum_{\xi\in\Irr(\G)}|b_{\xi}(\cl)|d_\xi\right) x^{1/2}\,\log(\mathcal{Q}_{\A,K}(x+\varsigma T))\nonumber\\
 &\ll_{\ST_F(\A)} & C_{\partial \Omega}^{1/2}\mathscr A_\G(\cl)
       \,T^{q/2}\,x^{1/2}\,\log(\mathcal{Q}_{\A,K}(x+\varsigma T)).
\end{eqnarray}
Since $C_{\partial\Omega}\ge1$, $T^{q/2}\ge1$, and
$\mathscr{A}_\G(\cl)-\frac{|\cl|}{|\G|}\le\mathscr{A}_\G(\cl)$,
the upper bound in \eqref{eq:5} is absorbed into the bound in \eqref{eq:6}. Combining \eqref{eq:4} and \eqref{eq:6}, equation \eqref{eq:3} yields
\begin{eqnarray*}
\left|\sum_{N\p\le x}\one_\Omega(\tilde{g}_\p)\one_\cl(\sigma_\p)-\frac{|\cl|}{|\G|}\mu_{\ST}(\Omega)\pi_F^{S}(x)\right|&\ll&\frac{|\cl|}{|\G|} \frac{x}{\log x}\frac{C_{\partial \Omega}}{T^{d-q}}\nonumber\\
&+& C_{\partial \Omega}^{1/2}\mathscr{A}_\G(\cl)T^{q/2}x^{1/2}\,\log(\mathcal{Q}_{\A,K}(x+\varsigma T))
\end{eqnarray*}
To balance the error terms, we use $\mathscr{A}_\G(\cl)\le|\cl|^{1/2}$, $C_{\partial\Omega}^{1/2}\le C_{\partial\Omega}$, and choose
 $$T=\left(\frac{\sqrt{|\cl|}\sqrt x} {|\G|\,\log({\cal Q}_{\A,K}x)\log x}\right)^{2/(2d-q)}.$$
With this choice, we have $T\le  x$. Hence, 
$ \log\bigl(\mathcal Q_{\A,K}(x+\varsigma T)\bigr)\asymp_{\ST_F(\A)}\log({\cal Q}_{\A,K}x)$ uniformly in $K$.
Note that
$$
 |\pi_F(x)-\pi_F^S(x)|\le |S|
                    \ll_{F,\A}\log\mathcal Q_{\A,K}.
$$
Therefore, replacing $\pi_F^S(x)$ with $\pi_F(x)$ changes the main term by at most
$O_{F,\A}\left(\frac{|\cl|}{|\G|}\log\mathcal Q_{\A,K}\right)$.
This is absorbed into the balanced error since
$$
 \frac{|\cl|}{|\G|}\log\mathcal Q_{\A,K}
 \le C_{\partial \Omega}\sqrt{|\cl|}\sqrt x\,T^{q/2}\log({\cal Q}_{\A,K}x),
$$
for $x\ge3$, $T\ge1$,
$C_{\partial \Omega}\ge1$, $\frac{|\cl|}{|\G|}\le\sqrt{|\cl|}$, and
$\log({\cal Q}_{\A,K}x)\ge\log\mathcal Q_{\A,K}$.
\end{proof}
The following version allows the regions and extensions to vary.
\begin{thm}\thlabel{Unifeffec}
Let $F$ be a number field. Let $(K_j/F)_{j\in J}$ be a family of finite Galois extensions with Galois group $\G_j$, and let $\cl_j$ be a conjugation-invariant subset of $\G_j$. Let $\A/F$ be an abelian variety with a connected Sato-Tate group $\ST_F(\A)$, and let $\boldsymbol\Omega:=(\Omega_i)_{i\in I}$ be a family of measurable conjugation-invariant subsets of $\ST_F(\A)$ such that $\overline{\mathcal{M}}^q_g(\partial \boldsymbol\Omega)<\infty$, where $0< q\le d-1$ with $d=\dim\ST_F(\A)$. Assume that  \thref{conjMT} holds for $\A$  and \thref{conjGRH} holds for all characters of $\CST_{K_j}(\A)$. Set
$$E_{i,j}(x):=\left|\sum_{N\p\le x}\one_{\Omega_i}(\tilde{g}_\p)\one_{\cl_j}(\sigma_{\p,j})-\frac{|\cl_j|}{|\G_j|}\mu_{\ST}(\Omega_i)\pi_F(x)\right|.$$
Then for every $(i,j)\in I\times J$, $T\ge 1$, and $x\ge 3$, we have
\begin{equation*}
E_{i,j}(x)=O\left(\frac{|\cl_j|}{|\G_j|} \frac{x}{\log x}\frac{C_{\partial \boldsymbol\Omega}}{T^{d-q}}+ C_{\partial \boldsymbol\Omega}^{1/2}\mathscr{A}_{\G_j}(\cl_j)T^{q/2}x^{1/2}\log(\mathcal{Q}_{\A,K_j}(x+\varsigma T))\right)
\end{equation*}
where $C_{\partial\boldsymbol\Omega}$ is defined in \eqref{eqboundaryfamily} and $\varsigma$ is the constant from \thref{wieghtbound}. The implied constant depends only on the geometry of the Sato-Tate group, the field $F$, and the abelian variety $\A$.
\end{thm}
\begin{proof}
Fix $T\ge 1$. \thref{BSuniform} together with the averaging argument of \thref{BScor1} provides a family of real-valued  combinations of irreducible characters
$$Q_{T,\Omega_{i}}^\pm=\sum_{c(\lambda)< T^2}a_{\lambda,T}^{\pm,i}\;\chi_\lambda,$$
such that
$$\quad Q_{T,\Omega_{i}}^-\le \one_{\Omega_i}\le Q_{T,\Omega_{i}}^+\quad \text{for every}\; i\in I.$$
For each $j\in J$, write $\one_{\cl_j}=\dsum_{\xi\in\Irr(\G_j)}b_{\xi}(\cl_j)\xi$. Following the proof of \thref{EffCST}, {\it mutatis mutandis}, one obtains the desired result. 
\end{proof}
Next, we establish a uniform effective Chebotarev–Sato–Tate estimate for prime-dependent regions.
\begin{thm}\thlabel{MovingCST}
Let $F$ be a number field. Let $(K_j/F)_{j\in J}$ be a family of finite Galois extensions with Galois group $\G_j$, and let $\cl_j\ne\varnothing$ be a conjugation-invariant subset of $\G_j$. Let $\A/F$ be an abelian variety with a connected Sato–Tate group $\ST_F(\A)$, and let $\boldsymbol\Omega=(\Omega_u)_{u\in I}$ be an increasing family of measurable conjugation-invariant subsets of $\ST_F(\A)$, indexed by an interval $I\subseteq\R$, and suppose that the family of boundaries satisfies $\overline{\mathcal{M}}^q_g(\partial\boldsymbol \Omega)<\infty$, where $0< q\le d-1$ and $d=\dim\ST_F(\A)$. Assume that \thref{conjMT} holds for $\A$ and \thref{conjGRH} holds for every irreducible character of $\CST_{K_j}(\A)$. Suppose that
$$
 \mu_{\ST}(\Omega_v\setminus\Omega_u)\le L_{\boldsymbol{\Omega}}(v-u)
                  \qquad \text{for every }u\le v,\,u,v\in I.
$$
Let $\varepsilon:=\varepsilon_X:[X,2X]\to I$ be Lipschitz with constant at most $V_X/X$. 
Set 
$$D_j(X):=\left|\sum_{X<N\p\le2X}\one_{\Omega_{\varepsilon(N\p)}}(\tilde{g}_\p)\one_{\cl_j} (\sigma_{\p,j})- \frac{|\cl_j|}{|\G_j|}\sum_{\substack{X<N\p\le2 X \\ \p\notin S_j}} \mu_{\ST}(\Omega_{\varepsilon(N\p)})\right|. $$
For every integer $R\ge1$ and every $T\ge1$, $X\ge3$, we have
\begin{eqnarray*}
D_j(X)=O\left( \frac{|\cl_j|}{|\G_j|}\frac{X}{\log X}\left(\frac{C_{\partial\boldsymbol\Omega}}{T^{d-q}}
+\frac{L_{\boldsymbol{\Omega}} V_X}{R}\right)+ R\,\mathscr A_{\G_j}(\cl_j)C_{\partial\boldsymbol\Omega}^{1/2}\sqrt X\,T^{q/2} \mathcal \log({\cal Q}_{\A,K_j}(X+\varsigma T))\right),
\end{eqnarray*}
where $C_{\partial\boldsymbol\Omega}$ is defined in \eqref{eqboundaryfamily} and $\varsigma$ is the constant from \thref{wieghtbound}. The implied constant depends only on the geometry of the Sato-Tate group, the field $F$, and the abelian variety $\A$.
\end{thm}
\begin{proof}
Let $j\in J$. Partition  $(X,2X]$ into $R$ equal half-open intervals
$$
 t_k=X+\frac{kX}{R},\qquad
 I_k=(t_{k-1},t_k]
$$
with $1\le k\le R$.  Then, one may write
$$\#\{X<N\p\le 2X\; :\; \p\notin S_j,\, \tilde{g}_\p\in\Omega_{\varepsilon(N\p)},\;
   \sigma_{\p,j}\in\cl_j \}=\sum_{k=1}^R\mathcal N_k$$
where  $\mathcal N_k:=\#\{N\p\in I_k\, :\, \p\notin S_j,\;
   \widetilde g_\p\in\Omega_{\varepsilon(N\p)},\;\; \sigma_{\p,j}\in\cl_j\}$, 
and $$\dsum_{\substack{X<N\p\le 2X\\ \p\notin S_j}}\mu_{\ST}(\Omega_{\varepsilon(N\p)})=\dsum_{k=1}^{R}\mathcal V_k\quad\text{where}\quad V_k=\sum_{\substack{N\p\in I_k\\ \p\notin S_j}}\mu_{\ST}(\Omega_{\varepsilon(N\p)}).$$
For $1\le k\le R$, by the continuity of $\varepsilon$, one may define $u_k^-:=\displaystyle\min_{t\in [t_{k-1},t_k]}\varepsilon(t)$ and $u_k^+:=\displaystyle\max_{t\in [t_{k-1},t_k]}\varepsilon(t)$.
For $t\in I_k$, the nesting assumption gives $\Omega_{u_k^-}\subseteq\Omega_{\varepsilon(t)} \subseteq\Omega_{u_k^+}$; hence, $\one_{\Omega_{u_k^-}}\le \one_{\Omega_{\varepsilon(t)}}\le \one_{\Omega_{u_k^+}}$. Therefore,
\begin{equation}\label{eqcountbracket1}
\underbrace{\sum_{N\p\in I_k}\one_{\Omega_{u_k^-}}(\tilde{g}_\p)\one_{\cl_j}(\sigma_{\p,j})}_{E^-_k}\le \mathcal{N}_k\le\underbrace{\sum_{N\p\in I_k}\one_{\Omega_{u_k^+}}(\tilde{g}_\p)\one_{\cl_j}(\sigma_{\p,j})}_{E^+_k}
\end{equation}
and 
\begin{equation}\label{eqcountbracket2}
    \pi_k^{S_j}\; \mu_{\ST}(\Omega_{u_k^-})\le \mathcal{V}_k\le  \pi_k^{S_j}\;\mu_{\ST}(\Omega_{u_k^+})
\end{equation}
where $\pi_k^{S_j}=\#\{N\p\in I_k : \p\notin S_j\}$. On the other hand,  as $\varepsilon$ is Lipschitz, we have $u_k^+-u_k^-\le \frac{V_X}{X}(t_k-t_{k-1})=\frac{V_X}{R}$. Hence
$$
 0\le\mu_{\ST}(\Omega_{u_k^+})-\mu_{\ST}(\Omega_{u_k^-})
 =\mu_{\ST}(\Omega_{u_k^+}\setminus\Omega_{u_k^-})
 \le\frac{L_{\boldsymbol\Omega} V_X}{R}.
$$
Consequently, by \eqref{eqcountbracket2}, one gets
\begin{equation}\label{eqvolumebracket}
 0\le\mu_{\ST}(\Omega_{u_k^+})\pi_k^{S_j}-\mathcal V_k
       \le\frac{L_{\boldsymbol\Omega} V_X}{R}\pi_k^{S_j},\qquad
 0\le\mathcal V_k-\mu_{\ST}(\Omega_{u_k^-})\pi_k^{S_j}
       \le\frac{L_{\boldsymbol\Omega} V_X}{R}\pi_k^{S_j}.
\end{equation}
By a similar argument to the one in the proof of \thref{Unifeffec}, we have 
$$E^\pm_k-\frac{|\cl_j|}{|\G_j|}\mu_{\ST}(\Omega_{u_k^\pm})\pi_k^{S_j}=O\left(\frac{|\cl_j|}{|\G_j|} \frac{C_{\partial\boldsymbol{\Omega}}}{T^{d-q}}\pi_k^{S_j}+C_{\partial\boldsymbol{\Omega}}^{1/2}\mathscr{A}_{\G_j}(\cl_j)T^{q/2}X^{1/2}\log(\mathcal{Q}_{\A,K_j}(X+\varsigma T))\right).$$
Equations \eqref{eqcountbracket2} and \eqref{eqvolumebracket} therefore lead to
\begin{equation*}
 \begin{split}
 \mathcal{N}_k-\frac{|\cl_j|}{|\G_j|}\mathcal V_k
 &\le E^+_k-\frac{|\cl_j|}{|\G_j|}\mu_{\ST}(\Omega_{u_k^+})\pi_k^{S_j}
             +\frac{|\cl_j|}{|\G_j|}(\mu_{\ST}(\Omega_{u_k^+})\pi^{S_j}_k-\mathcal V_k)\\
 &\ll \frac{|\cl_j|}{|\G_j|} \frac{C_{\partial\boldsymbol{\Omega}}}{T^{d-q}}\pi_k^{S_j}+C_{\partial\boldsymbol{\Omega}}^{1/2}\mathscr{A}_{\G_j}(\cl_j)T^{q/2}X^{1/2}\log(\mathcal{Q}_{\A,K_j}(X+\varsigma T))
             +\frac{|\cl_j|}{|\G_j|}\frac{ L_{\boldsymbol\Omega} V_X}{R}\pi_k^{S_j}.
 \end{split}
\end{equation*}
Similarly, for the lower discrepancy, we have
\begin{equation*}
\frac{|\cl_j|}{|\G_j|}\mathcal{V}_k-\mathcal N_k
 \ll \frac{|\cl_j|}{|\G_j|} \frac{C_{\partial\boldsymbol{\Omega}}}{T^{d-q}}\pi_k^{S_j}+C_{\partial\boldsymbol{\Omega}}^{1/2}\mathscr{A}_{\G_j}(\cl_j)T^{q/2}X^{1/2}\log(\mathcal{Q}_{\A,K_j}(X+\varsigma T))+\frac{|\cl_j|}{|\G_j|}\frac{ L_{\boldsymbol\Omega} V_X}{R}\pi_k^{S_j}.
\end{equation*}
Therefore
\begin{equation}\label{eqblockerror}
 \left| \mathcal{N}_k-\frac{|\cl_j|}{|\G_j|}\mathcal V_k\right|
 \ll \frac{|\cl_j|}{|\G_j|} \frac{C_{\partial\boldsymbol{\Omega}}}{T^{d-q}}\pi_k^{S_j}+C_{\partial\boldsymbol{\Omega}}^{1/2}\mathscr{A}_{\G_j}(\cl_j)T^{q/2}X^{1/2}\log(\mathcal{Q}_{\A,K_j}(X+\varsigma T))+\frac{|\cl_j|}{|\G_j|}\frac{ L_{\boldsymbol\Omega} V_X}{R}\pi_k^{S_j}.
\end{equation}
Finally, note that
$$
 \sum_{k=1}^R\pi_k^{S_j}
 =\pi_F^{S_j}(2X)-\pi_F^{S_j}(X)
 \le\pi_F(2X)\ll_F\frac{X}{\log X}.
$$
Summing \eqref{eqblockerror} over $k$ yields
\begin{equation*}
 \begin{split}
 D_j(X)&\le\sum_{k=1}^R\left|\mathcal N_k-\frac{|\cl_j|}{|\G_j|}\mathcal V_k\right|\\
 &\ll\frac{|\cl_j|}{|\G_j|}\left(\frac{C_{\partial\boldsymbol{\Omega} }}{T^{d-q}}+\frac{L_{\boldsymbol\Omega} V_X}{R}\right)
\frac{X}{\log X}+RC_{\partial\boldsymbol{\Omega}}^{1/2}\mathscr{A}_{\G_j}(\cl_j)T^{q/2}X^{1/2}\log(\mathcal{Q}_{\A,K_j}(X+\varsigma T)),
 \end{split}
\end{equation*}
which completes the proof.
\end{proof}
\subsection{Murty's conjecture}\label{subs 4.3}
 Let $f$ be a non-CM normalized newform of weight $2$ and level $N$,  with a quadratic coefficient field and no nontrivial inner twists. Write $a(p)$ for its $p$-th coefficient and set
 $$\Pi(x)=\#\{p\le x:p\nmid N,\ a(p)\in\Q\}.$$
 Building on the earlier work of Lang and Trotter \cite{lang2006}, Murty \cite{murty1999} proposed the following conjecture:
\begin{conjecture}
Let $f$ be a non-CM normalized newform of weight 2 and level $N$ with quadratic coefficient field and without inner twists. Then 
$$
\Pi(x)\sim c_f\frac{x^{1/2}}{\log x},$$
with $c_f$ a constant that depends on $f$.
\end{conjecture}
In \cite{van2017}, Van Hirtum presents a heuristic model that supports the conjecture above and makes $c_f$ explicit. His model relies on assumptions combining the independence of
finite places (governed by the Chebotarev distribution) and archimedean conditions (governed by the Sato-Tate distribution) with an approximation for shrinking regions. The quantitative prime-dependent region estimate in \thref{MovingCST} addresses this uniformity problem by providing an explicit bound for the discrepancy. To elaborate, since $f$ has no inner twists, its nebentypus is trivial. Its coefficient field $K_f$ is therefore totally real. Let $D$ be the positive squarefree integer such that $K_f:=\Q(\sqrt{D})$. Let us denote by $\bar{\cdot}$
the unique non-trivial element of $\Gal(K_f/\Q)$. Define
$$Z_p:=a(p)-\overline{a(p)}.$$
Let $\A_f$ be the abelian surface defined over $\Q$ attached to $f$ by the Shimura construction \cite[Section 6.6]{diamond2005} with absolute conductor $N^2$. For a prime $\ell$, let
$$\rho_{\A_f,\ell}:G_{\Q}\longrightarrow \GSp_4(\overline{\Q}_\ell)$$
be the $\ell$-adic representation attached to $\A_f$.
Since $f$ has no inner twists, ${\rm End}^0_{\overline{\Q}}(\A_f)={\rm End}^0_{\overline{\Q}}(\A_f)=K_f$ (cf. \cite[Theorem 5]{ribet1979}). In the terminology of \cite{fite2012}, $\A_f$ is of Galois type $[1,\R\times\R]$. Therefore, its Sato-Tate group is $\ST_{\Q}(\A_f)=\SU(2)\times\SU(2)$. Furthermore, \thref{conjMT} is known to hold for abelian varieties of dimension less than 3. For a prime $p\nmid N$, we write $\tilde{g}_p$ for the normalized image of the arithmetic Frobenius ${\rm Frob}_p$ of the absolute Galois group $G_\Q$ under the $\ell$-adic representation attached to $\A_f$, as defined in Subsection \ref{Sec 3.1}. The normalized $p$-Weil polynomial $L_{\A_f,p}$ of $\A_f$ factors as
$$
L_{\A_f,p}(T):=\det(1-\tilde{g}_pT)=\left(T^2-\frac{a(p)}{\sqrt{p}}T+1\right)\left(T^2-\frac{\overline{a(p)}}{\sqrt{p}}T+1\right).
$$

Let $O_f$ denote the ring of integers of $K_f$. For every prime $\ell$, fix an integral realization
$$
\rho_{f,\ell}:G_{\mathbb Q}\longrightarrow
\mathrm{GL}_2(\mathcal O_f\otimes_{\mathbb Z}\mathbb Z_\ell)
$$
of the $\ell$-adic Galois representation attached to \(f\). For each positive integer $m$, let
$$
\rho_m:G_{\mathbb Q}\longrightarrow
\mathrm{GL}_2(\mathcal O_f/m\mathcal O_f)
$$
be the representation obtained, via the Chinese remainder theorem, from the reductions of $\rho_{f,\ell}$ modulo $\ell^{v_\ell(m)}$ for $\ell\mid m$. Then, we have 
$$
\tr\rho_m({\rm Frob}_p)\equiv a(p)\bmod{m{\cal O}_f}\qquad(p\nmid mN).
$$
 Let $\mathscr{G}_m:=\mathrm{Im}\;\rho_m$, and define the conjugation-invariant subset
$$\cl_m:=\{\gamma\in\G_m\;:\;\tr(\gamma)-\overline{\tr(\gamma)}=0\bmod{m O_f}\}.$$
Let $K_m:=\overline{\mathbb{Q}}^{\mathrm{Ker}\;\rho_m}$ be the field cut out by $\rho_m$. Then $K_m/\Q$ is a finite Galois extension with Galois group $\G_m$. Moreover, $K_m$ is unramified outside the primes dividing $mN$. Since the ring $\mathcal O_f/m\mathcal O_f$ has $m^2$ elements, we have $|\G_m|\le m^8$. Hence,
\begin{equation}\label{eqconductorbis}
 M(K_m/\mathbb Q)
   =2|\G_m|\prod_{p\in P(K_m/\mathbb Q)}p
   \le 2m^8\operatorname{rad}(mN)
   \le 2m^9\operatorname{rad}(N).
\end{equation}
where $\operatorname{rad}(N)=\displaystyle\prod_{p\mid N}p$. For $\varepsilon>0$, set
$$
\Omega_\varepsilon=\{(\gamma_1,\gamma_2)\in \SU(2)\times\SU(2):|\Phi(\gamma_1,\gamma_2)|<\varepsilon\},\quad\text{with}\quad\Phi(\gamma_1,\gamma_2)=\frac{1}2(\tr(\gamma_1)-\tr(\gamma_2)).
$$
For a prime $p\nmid mN$, we write $\sigma_{p,m}=\left(\dfrac{K_m/\Q}{p}\right)$ for the Artin symbol relative to $K_m/\Q$ at $p$.
For each fixed $m\ge 1$ and every prime $p\nmid mN$, we have
\begin{eqnarray*}
a(p)\in\Q &\Leftrightarrow& Z_p=0\\&\Leftrightarrow& \tilde{g}_p\in \Omega_{\frac{m\sqrt{D}}{4\sqrt{p}}}\quad\text{and}\quad \sigma_{p,m}\in \cl_m.
\end{eqnarray*}
Indeed $Z_p\in\sqrt D\Z$. Moreover, $\mathcal O_f\cap\sqrt D\Q=\sqrt D\Z$, which follows from the integral basis $1,\sqrt D$ if $D\equiv2,3\bmod4$. For $D\equiv1\bmod4$, write an element of $\mathcal O_f$ as $a+b(1+\sqrt D)/2$ for $a,b\in\Z$, and impose trace zero yields $b=-2a$. Consequently, $Z_p\in m\mathcal O_f$ forces $Z_p\in m\sqrt D\Z$, including at $p=2$. On the other hand
$$
 |\Phi(\widetilde g_p)|=\frac{|Z_p|}{2\sqrt p}
 <\frac{m\sqrt D}{4\sqrt p}
 \quad\Longleftrightarrow\quad |Z_p|<\frac{m\sqrt D}{2}.
$$
Zero is the only element of $m\sqrt D\Z$ in this interval.
Thus, Murty’s conjecture reduces to a quantitative problem of Chebotarev-Sato-Tate type involving varying Galois extensions and prime-dependent regions. We now examine what follows from the standard analytic assumptions of \thref{conjGRH} on the $L$-functions attached to irreducible characters of the associated Chebotarev–
Sato–Tate group $\CST_{K_m}(\A_f):=\ST_\Q(\A_f)\times\G_m$.  

The family $\boldsymbol{\Omega}:=(\Omega_\varepsilon)_{\varepsilon\in (0,\infty)}$ consists of measurable conjugation-invariant subsets of $\ST_\Q(\A_f)$. The following lemma gives a uniform bound for the upper 5-dimensional Minkowski content of their boundaries, allowing us to apply \thref{MovingCST}.
\begin{lemma}\thlabel{MurtyTubes}
Fix a bi-invariant Riemannian metric on $\ST_\Q(\A_f)$, and let
$L_\Phi=\displaystyle\sup_{\ST_\Q(\A_f)}\|\nabla \Phi\|$. For all $\varepsilon,\varepsilon',\Delta>0$, we have
\begin{enumerate}[label=(\roman*)]
    \item $\mu_{\ST}((\partial\Omega_\varepsilon)^{(\Delta)})\le\frac{64L_\Phi}{3\pi^2}\Delta.$

\item $|\mu_{\ST}(\Omega_\varepsilon)-\mu_{\ST}(\Omega_{\varepsilon'})|\le\frac{32}{3\pi^2}|\varepsilon-\varepsilon'|,$ 
\item $ \mu_{\ST}(\Omega_\varepsilon)=\frac{32}{3\pi^2}\varepsilon+O(\varepsilon^2) \qquad(\varepsilon\to0^+).$
\end{enumerate}
In particular, $C_{\partial\boldsymbol\Omega}\le\max\{1,64L_\Phi/(3\pi^2)\}$.
\end{lemma}
\begin{proof}
The fixed bi-invariant metric on $\ST_\Q(\A_f)$ endows $(\ST_\Q(\A_f),d_g)$ with the structure of a compact, connected, real analytic Riemannian manifold of dimension 6. The continuity of $\Phi$ gives
\begin{equation}\label{eqomegaepsilon}
\partial\Omega_\varepsilon\subset
\{\Phi=\varepsilon\}\cup\{\Phi=-\varepsilon\}.
\end{equation}
For $0<\varepsilon<2$, this inclusion is an equality,
and the two level sets are compact smooth hypersurfaces of dimension $5$ (the critical values of $\Phi$ are $-2,0,2$, so $\pm\varepsilon$ are regular values when $0<\varepsilon<2$). By Weyl's integration formula, the pushforward of the normalized Haar measure on $\SU(2)$ under $\frac{1}{2}\tr\,:\SU(2)\longrightarrow [-1,1]$ has density 
$$w(t):=\begin{cases}
    \frac{2}{\pi}\sqrt{1-t^2}&\text{if}\; t\in [-1,1],\\
    0& \text{otherwise.}
\end{cases}$$
Since the Haar probability measure on $\ST_\Q(\A_f)$ is the product of the Haar probability measures on its two factors, for every Borel set $I\subseteq\mathbb R$,
\begin{equation}\label{eqpushfor}
\mu_{\mathrm{ST}}(\Phi^{-1}(I))=\int_I r(z)\,\d z,
\end{equation}
where
$$
 r(z)=\int_{\mathbb R}w(t)w(t-z)\,\d t,
$$
 which is even, as is $w$. In particular,
\begin{equation}\label{eqomega}
 \mu_{\mathrm{ST}}(\Omega_\varepsilon)
       =\int_{-\varepsilon}^{\varepsilon}r(z)\,\d z.
\end{equation}
On the other hand, the Cauchy-Schwarz inequality gives
\begin{equation}\label{eqdens}
 0\le r(z)\le \frac4{\pi^2}\int_{-1}^1(1-t^2)\,\d t
   =\frac{16}{3\pi^2}=r(0).
\end{equation}
The gradient bound yields
$$
 |\Phi(x)-\Phi(y)|\le L_\Phi d_g(x,y)
$$
for every  $x,y\in \ST_\Q(\A_f)$. On account of \eqref{eqomegaepsilon} and the above displayed formula, for $\Delta>0$ we have
$$
(\partial\Omega_\varepsilon)^{(\Delta)}\subset \Phi^{-1}
(I(\varepsilon,L_\Phi\Delta))
\cup
 \Phi^{-1}(I(-\varepsilon,L_\Phi\Delta))
$$
where $I(\pm\varepsilon,L_\Phi\Delta)=(\pm\varepsilon-L_\Phi\Delta,\pm\varepsilon+L_\Phi\Delta)$. By \eqref{eqdens} together with \eqref{eqpushfor}, we infer immediately (i). For (ii), assume $\varepsilon'\le\varepsilon$. From
\eqref{eqomega} we see
$$
 \mu_{\mathrm{ST}}(\Omega_\varepsilon)
       -\mu_{\mathrm{ST}}(\Omega_{\varepsilon'})
 =2\int_{\varepsilon'}^{\varepsilon}r(z)\,\d z
 \le \frac{32}{3\pi^2}\,(\varepsilon-\varepsilon'),
$$
which yields (ii). As for (iii), we have
$$|r(z)-r(0)|\le  \sup_{t\in \R}|w(t)|\int_\R |w(t-z)-w(t)|\d t\le |z| \sup_{t\in \R}|w(t)|\int_\R|w'(t)|\d t\le \frac{8}{\pi^2}|z|.$$
Integrating the above displayed inequality over $[-\varepsilon,\varepsilon]$ yields
\begin{eqnarray*}
\left|
\mu_{\ST}(\Omega_\varepsilon)
-\frac{32}{3\pi^2}\varepsilon
\right|
&\le& \int_{-\varepsilon}^{\varepsilon}|r(z)-r(0)|\,\d z\\
&\le& \frac8{\pi^2}\varepsilon^2,
\end{eqnarray*}
which proves (iii). The linear coefficient agrees with \cite[Lemma 4.5]{van2017}, and our argument also gives an explicit
remainder.
\end{proof}
For $X>\max\{N,3\}$ and $1\le m\le X$, no prime in
$(X,2X]$  divides $mN$. Define the signed discrepancy
\begin{equation}\label{DMurty}
 D_m(X):=
 \#\{X<p\le2X:\sigma_{p,m}\in \cl_m,\;\widetilde g_p\in\Omega_{\varepsilon_{m,p}}\}-\frac{|\cl_m|}{|\G_m|}\sum_{X<p\le2X}\mu_{\mathrm{ST}}(\Omega_{\varepsilon_{m,p}})
\end{equation}
where $\varepsilon_{m,p}:=\frac{m\sqrt{D}}{4\sqrt{p}}$. Let $\varepsilon_X: [X,2X]\longrightarrow (0,\infty)$ be the function  defined by $\varepsilon_X(t)=\frac{m\sqrt D}{4\sqrt t}$. We have 
$$
 |\varepsilon_X'(t)|=\frac{m\sqrt D}{8t^{3/2}}
                \le\frac{m\sqrt D}{8X^{3/2}}\qquad\text{for }\; X\le t\le2X.
$$
Therefore, $\varepsilon_X$ is Lipschitz with a Lipschitz constant at most $\frac{m\sqrt{ D}}{8X^{3/2}}$. For $2\le T\le X$ and $R\in\Z_{\ge 1}$, assuming that \thref{conjGRH} holds for every irreducible character of $\CST_{K_m}(\A_f)$, \thref{MovingCST} gives
$$
|D_m(X)|\ll_f
\frac{|\cl_m|}{|\G_m|}\frac{X}{T\log X}
+\frac{F_m\sqrt X}{R\log X}
+R\mathscr{A}_m\sqrt X\,T^{5/2}\log({\cal Q}_{\A_f,K_m}(X+\varsigma T)),
$$
where $F_m:=\frac{m|\cl_m|}{|\G_m|}$, which we shall call the finite-place factor at modulus $m$, $\varsigma$ is the constant from \thref{wieghtbound}, and $\mathscr{A}_m:=\mathscr{A}_{\G_m}(\cl_m)$. The factors $F_m$ converge to $\widehat{F}_f$ by divisibility, that is, for every $\eta>0$, there is an integer $m_0$ such that $|F_m-\widehat F_f|<\eta$ whenever $m_0\mid m$, which we denote by 
\begin{equation}\label{eq:limdivisibility}
\lim_{m\to|\infty} F_m=\widehat{F}_f.
\end{equation}
The constant $\widehat F_f$ (which we shall call the finite-place factor) is strictly positive and has an explicit Euler product formula (see Appendix \ref{app:finite-place-factor}). Furthermore, uniformly for all positive integers $m$, we have $F_m=O_f(1)$ and $|\G_m|\ll_f m^7$ (see  \thref{apprem}). Consequently,

$$
\frac{|\cl_m|}{|\G_m|}
\ll_f\frac1m,
\qquad
|\cl_m|
=\frac{F_m}{m}|\G_m|
\ll_f m^6,
\qquad
\mathscr{A}_m\le|\cl_m|^{1/2}\ll_f m^3.
$$
Moreover, \eqref{eqconductorbis} gives
$${\cal Q}_{\A_f,K_m}\ll_f m^9.$$
Take $R=1$ and balance the first and third terms by choosing
$$
T=\left(\frac{X}{m^8(\log X)^4}\right)^{1/7},
$$
provided $T\ge2$. Thus, for $m\le  \frac{X^{1/8}}{2^{7/8}(\log X)^{1/2}}$ and $X\ge \max\{N,3\}$ we have
\begin{equation}\label{eqoptimDm}
|D_m(X)|\ll_f \frac{m^{1/7}X^{6/7}}{(\log X)^{3/7}},
\end{equation}
which is larger than the scale $X^{1/2}/\log X$ required for Murty's predicted asymptotic. Nevertheless, we obtain the following conditional upper bound.
\begin{proposition}\thlabel{MurtyUpper}
\sloppy Assume \thref{conjGRH} for every nontrivial irreducible character of $\ST_\Q(\A_f)$. Then we have
$$
 \Pi(x)\ll_f\frac{x^{6/7}}{(\log x)^{3/7}}.
$$
\end{proposition}
\begin{proof}
Let $X>\max\{N,3\}$. For $m=1$, \eqref{DMurty} may be written as 
$$D_1(X)=\Pi(2X)-\Pi(X)
-\sum_{X<p\le2X}
\mu_{\ST}
\!\left(\Omega_{\sqrt D/(4\sqrt p)}\right).
$$
By \thref{MurtyTubes}(iii), we have
$$
\sum_{X<p\le2X}
\mu_{\ST}
\!\left(\Omega_{\sqrt D/(4\sqrt p)}\right)
\ll_f\sum_{X<p\le2X}p^{-1/2}
\ll_f\frac{\sqrt X}{\log X}.
$$
Hence, from \eqref{eqoptimDm}, it follows that
\begin{equation}\label{eq2XX}
\Pi(2X)-\Pi(X)\ll_f \frac{X^{6/7}}{(\log X)^{3/7}}.
\end{equation}
For sufficiently large $x$, put $t=\left\lfloor\frac{1}{2}\log_2 x\right\rfloor$. 
Set $X_j:=x/2^{j}$ for $0\le j\le t$. Then $\sqrt x\le X_t<2\sqrt x$.  Applying \eqref{eq2XX} to $(X_j,2X_j]$ for $1\le j\le t$ and using $\Pi(X_t)\le2\sqrt x$ in the dyadic identity 
$$
\Pi(x)=\Pi(X_t)+\sum_{j=1}^{t}(\Pi\!(X_{j-1})-\Pi(X_j)),
$$
one obtains the desired bound.
\end{proof}

We retain the preceding notation and describe a probabilistic model for the discrepancy estimate conjectured below. Let $B>0$, and let $m=m(X)$ be a positive integer-valued function such that $m(X)\le(\log X)^B$ for all sufficiently large $X$, and every fixed positive integer eventually divides $m(X)$. By the limit-by-divisibility defining $\widehat F_f$ in \eqref{eq:limdivisibility}, we have $F_m\to\widehat F_f$; in particular, $F_m=O_f(1)$ holds along this function. For each $X$, we model the indicators of the joint event $\widetilde g_p\in\Omega_{\varepsilon_{m,p}}$ and $\sigma_{p,m}\in\cl_m$ by Bernoulli random variables $B_p$, indexed by primes $X<p\le2X$. We assume that they are independent. Motivated by the hybrid distribution predicted by \thref{conjCST}, we assign to $B_p$ the success probability
$$
\mathbb P(B_p=1)=r_{m,p}:=\frac{|\cl_m|}{|\G_m|}
\mu_{\ST}(\Omega_{\varepsilon_{m,p}}).
$$
Define the model count by $\mathcal N_X:=\dsum_{X<p\leq2X}B_p$. The expectation equals the product-density main term
$$
M_X:=\mathbb E[\mathcal N_X]
=\frac{|\cl_m|}{|\G_m|}
\sum_{X<p\leq2X}
\mu_{\ST}(\Omega_{\varepsilon_{m,p}}).
$$
By \thref{MurtyTubes}(iii), we have
$$
M_X=\frac{8\sqrt D}{3\pi^2}F_m\sum_{X<p\leq2X}p^{-1/2}+O_f\left(\frac{mF_m}{\log X}\right)\ll_f\frac{\sqrt X}{\log X}.
$$
Here we use the boundedness of $F_m$ along the chosen function and the bound $m\leq(\log X)^B$. Independence gives the variance bound
$$
\operatorname{Var}(\mathcal N_X)=\sum_{X<p\leq2X}r_{m,p}(1-r_{m,p})
\leq M_X\ll_f\frac{\sqrt X}{\log X}.
$$
Consequently, Chebyshev's inequality gives, for every fixed $\eta>0$,
$$
\mathbb P\left(|\mathcal N_X-M_X|>\eta\frac{\sqrt X}{\log X}\right)\ll_f\frac{\log X}{\eta^2\sqrt X}.
$$
Since the right-hand side tends to zero, we infer that
$$
\mathcal N_X-M_X
=o_{\mathbb P}\left(\frac{\sqrt X}{\log X}\right).
$$
This motivates the following conjecture.
\begin{conjecture}\thlabel{Murtyoptimalerror}
Let $B>0$ and $m=m(X)$ be a positive integer-valued function such that $m\le(\log X)^B$ for all sufficiently large $X$, and every fixed positive integer $n$ divides $m(X)$ for all sufficiently large $X$. Then the discrepancy defined in \eqref{DMurty} satisfies
 $$
D_{m}(X)=o\left(\frac{\sqrt X}{\log X}\right)\qquad(X\to\infty).
 $$ 
\end{conjecture}
Such a cofinal function exists for every $B>0$. Indeed, for any $B>0$, take $m(X)=k(X)!$, where $k(X)$ is the largest integer satisfying $k(X)!\le(\log X)^B$. As $X\to\infty$, we have $k(X)\to\infty$, so every fixed positive integer divides $m(X)$ for all sufficiently large $X$. Then  \thref{propfinite-placeconstant} gives $F_m\to \widehat{F}_f$.
\begin{thm}
Assume \thref{Murtyoptimalerror}. Then
$$
 \Pi(x)\sim\frac{16\sqrt D\widehat F_f}{3\pi^2}\frac{\sqrt x}{\log x}\qquad\text{as}\;x\to\infty.
$$
\end{thm}
\begin{proof}
Choose a function $m=m(X)$ satisfying the conditions of \thref{Murtyoptimalerror}. For sufficiently large $X$, we have $m\le X$. Since $m$ is cofinal for divisibility, $F_m\to\widehat F_f$ as $X\to\infty$.
By \thref{MurtyTubes}(iii) we have
\begin{equation}\label{eqdyadicsum}
    \frac{|\cl_m|}{|\G_m|}\sum_{X<p\le2X}\mu_{\mathrm{ST}}(\Omega_{\varepsilon_{m,p}})
   =\frac{8\sqrt{D}F_m}{3\pi^2} \sum_{X<p\le2X}p^{-1/2}
             +O_f\left(\frac{mF_m}{\log X}\right).
\end{equation}
Consequently, we can decompose the discrepancy $D_m(X)$ as
\begin{eqnarray}\label{eqDiscsplit}
D_m(X)&=&\Pi(2X)-\Pi(X)
              -\frac{8\sqrt{D}}{3\pi^2}\widehat F_f \sum_{X<p\le2X}\frac{1}{p^{1/2}}\nonumber\\
 &+&\frac{8\sqrt{D}}{3\pi^2}(\widehat F_f-F_m)\sum_{X<p\le2X}\frac{1}{p^{1/2}}+O_f\left(\frac{mF_m}{\log X}\right).
\end{eqnarray}          
Note that
$$
 \frac{mF_{m}}{\log X} =o\left(\frac{\sqrt X}{\log X}\right).
$$
On the other hand, by the prime number theorem, we have
$$ \dsum_{X<p\le 2X}\frac{1}{p^{1/2}}\ll\frac{\sqrt X}{\log X}.$$ 
It follows that
$$\bigl(\widehat F_f-F_{m}\bigr)\sum_{X<p\le2X}\frac{1}{p^{1/2}}
 =o\!\left(\frac{\sqrt X}{\log X}\right).$$
Thus, \eqref{eqDiscsplit}, \eqref{eqdyadicsum}, and \thref{Murtyoptimalerror} yield
\begin{equation}\label{murty:dyadicmean}
 \Pi(2X)-\Pi(X)=\frac{8\sqrt{D}\widehat F_f}{3\pi^2} \sum_{X<p\le 2X}\frac{1}{p^{1/2}}+o\left(\frac{\sqrt X}{\log X}\right).
\end{equation}
From the prime number theorem and partial summation, we deduce that
$$
  \sum_{X<p\le 2X}\frac{1}{p^{1/2}}=2(\sqrt2-1)\frac{\sqrt X}{\log X}+o\!\left(\frac{\sqrt X}{\log X}\right).
$$
Substitution in \eqref{murty:dyadicmean} yields
\begin{equation}\label{murty:dyadic}
 \Pi(2X)-\Pi(X)=\frac{16\sqrt{D}\widehat{F}_f}{3\pi^2}(\sqrt2-1)\frac{\sqrt X}{\log X} +o\left(\frac{\sqrt X}{\log X}\right).
\end{equation}
For sufficiently large $x$, put $t=\left\lfloor\frac{3}{4}\log_2 x\right\rfloor$. 
Set $X_j:=x/2^{j}$ for $0\le j\le t$. Then we have $x^{1/4}\le X_t<2 x^{1/4}$, and
$$
\Pi(x)=\Pi(X_t)+\sum_{j=1}^{t}(\Pi(X_{j-1})-\Pi(X_j)).
$$
Since $\Pi(X_t)\le 2x^{1/4}$, we have 
$$\lim_{x\to\infty}\frac{\Pi(X_t)}{\sqrt{x}/\log x}=0.$$
For $1\le j\le t$, one has $\log X_j\ge \frac14\log x$. It then follows that for every fixed $j$, \eqref{murty:dyadic} implies
$$
\lim_{x\to\infty} \frac{\Pi(X_{j-1})-\Pi(X_j)}{\sqrt x/\log x}
       =\frac{16\sqrt{D}\widehat{F}_f}{3\pi^2}(\sqrt2-1)2^{-j/2}.
$$
Now, write 
$$\frac{\Pi(x)}{\sqrt x/\log x}
 =\dfrac{\Pi(X_t)}{\sqrt{x}/\log x}+\sum_{j= 1}^{t}\frac{\Pi(X_{j-1})-\Pi(X_j)}{\sqrt x/\log x}.$$
Define the summand to be zero for $j>t$. By \eqref{murty:dyadic}, we have 
$$\frac{\Pi(X_{j-1})-\Pi(X_j)}{\sqrt x/\log x}\ll_f 2^{-j/2}$$
Consequently, the dominated convergence theorem permits interchanging the limit and the sum, yielding
$$
 \lim_{x\to\infty}\frac{\Pi(x)}{\sqrt x/\log x}
 =\frac{16\sqrt{D}\widehat{F}_f}{3\pi^2}(\sqrt2-1)\sum_{j=1}^{\infty}2^{-j/2}=\frac{16\sqrt{D}\widehat{F}_f}{3\pi^2},
 $$
since $\dsum_{j=1}^{\infty}2^{-j/2}=1/(\sqrt2-1)$.
This proves the asserted asymptotic.
\end{proof}
\appendix
\section{Appendix: The finite-place factor in Murty's conjecture}
\label{app:finite-place-factor}

Van Hirtum proves in \cite[Corollary 5.10]{van2017} the finiteness of the limit-by-divisibility of the finite-place factor at modulus $m$. For empirical evidence, he assumes \cite[Assumption 7.2]{van2017}, which allows him to express the factor in question as a product of local factors. For odd unramified primes with a large image, local factors were explicitly computed in \cite[Propositions 5.5–5.8 and Lemma 5.9]{van2017}. However, for possibly exceptional and ramified primes, this was accomplished via a result of Billerey and Dieulefait in \cite[Lemma 5.9]{van2017}. In what follows, we aim to prove the strict positivity and derive an Euler product formula for this factor in the spirit of \cite[Lemma 12]{Cojocaru2016} via a local counting argument and adelic openness \cite[Theorem 2.3.1]{loeffler2017}.

We retain the notations from  \S\ref{subs 4.3}. Let $f$ be a normalized non-CM newform of weight $2$ and level $N$ with a trivial nebentypus, a real quadratic coefficient field $K_f$, and no nontrivial inner twists. Let
$$
\widehat\rho_f:=\varprojlim_m\rho_m:G_\Q\longrightarrow
\GL_2(\mathcal O_f\otimes_\Z\widehat\Z),
$$
be the adelic representation attached to $f$ and $\Gamma:=\operatorname{Im}\widehat\rho_f$ its image.
For each rational prime $\ell$, set
$$
\mathcal O_{f,\ell}:=\mathcal O_f\otimes_\Z\Z_\ell,
\qquad
H_\ell:=
\left\{g\in\GL_2(\mathcal O_{f,\ell}):
\det g\in\Z_\ell^\times\right\},
\qquad
\widehat H:=\prod_\ell H_\ell.
$$
Since $\det\rho_{f,\ell}=\chi_\ell$ where $\chi_\ell: G_\Q\longrightarrow\Z_\ell^\times$ is the $\ell$-adic cyclotomic character, we have $\Gamma\subseteq\widehat H$.
By the adelic open-image theorem
\cite[Theorem 2.3.1]{loeffler2017}, $\Gamma$ is open in $\widehat H$. We may therefore choose a positive integer $m_f$, divisible by $2$ and by every prime ramified in $K_f$, such that
$$
\Gamma=\pi_{m_f}^{-1}(\G_{m_f}),
$$
where $\pi_{m}:\widehat H\longrightarrow\GL_2(\mathcal O_f/m\mathcal O_f)$ is the reduction map. Fix an integral basis $(1,\omega)$ for
$\mathcal O_f$, and write
\begin{equation}\label{eqA1}
\omega^2-u\omega+v=0,
\qquad u,v\in\Z.
\end{equation}
Then $\overline\omega=u-\omega$.

\begin{pro}\thlabel{propfinite-placeconstant}
The limit
$$
\widehat F_f:=\lim_{m\to\mid\infty}F_m
$$
exists and belongs to $(0,\infty)$. Moreover,
$$
\widehat F_f
=\widehat F_{m_f}\prod_{\ell\nmid m_f}\widetilde F_\ell,
\qquad
0<\widehat F_{m_f}<\infty,
$$
where
$$
\widetilde F_\ell=
\begin{cases}
\displaystyle
\frac{\ell^3(\ell^2+\ell+1)}
     {(\ell+1)(\ell^4-1)},
&\ell\text{ inert in }K_f,\\
\displaystyle
\frac{\ell^2(\ell^3+\ell^2-\ell-2)}
     {(\ell+1)(\ell^2-1)^2},
&\ell\text{ split in }K_f.
\end{cases}
$$
The factor $\widehat F_{m_f}$ is defined in \eqref{eqA12}. 
\end{pro}

We first establish the following technical lemma. 

\begin{lem}\thlabel{lemlocaldensity}
Let $\ell$ be a rational prime. For $g\in H_\ell$, define $t(g)\in\Z_\ell$ by
$$
\tr g=A(g)+\omega t(g),
\qquad A(g)\in\Z_\ell.
$$
Let $\mu_\ell$ be the normalized Haar measure on $H_\ell$.
For every compact open subset $U\subseteq H_\ell$, the limit
$$
L_\ell(U):=
\lim_{k\to\infty}\ell^k
\mu_\ell\left(\{g\in U:t(g)\in\ell^k\Z_\ell\}\right)
$$
exists and is finite. Moreover, if $I_2\in U$, then $L_\ell(U)>0$.
\end{lem}

\begin{proof}
For $k\geq1$, let $\pi_{\ell,k} : H_\ell\longrightarrow M_2(\Z_\ell[\omega]/\ell^k\Z_\ell[\omega])$ denote reduction modulo $\ell^k$, and set $H_\ell(k):=\pi_{\ell,k}(H_\ell)$. We claim that
$$H_\ell(k)=\{\bar g\in\GL_2\left(\Z_\ell[\omega]/\ell^k\Z_\ell[\omega]\right)\; :\; \det\bar g\in (\Z/\ell^k\Z)^\times\}.$$
Indeed, for $\bar g$ belonging to the right-hand side, choose lifts $g\in M_2(\Z_\ell[\omega])$ of $\bar g$ and $a\in\Z_\ell^\times$ of $\det \bar g$. Then we have $\delta:=\det g\equiv a\bmod{\ell^k\Z_\ell[\omega]}$. So, one may write $\delta=a(1+\ell^k r)$ for some $r\in\Z_\ell[\omega]$. The series 
$$
\dsum_{n=0}^\infty (-\ell^k r)^n
$$ 
converges in $\Z_\ell[\omega]$ and is the inverse of $1+\ell^k r$. Thus, $1+\ell^k r\in\Z_\ell[\omega]^\times$, and hence $\delta\in\Z_\ell[\omega]^\times$. Set $\varepsilon:=\delta a^{-1}$ and $g':=\mathrm{diag}(\varepsilon^{-1},1)g$, since $\varepsilon\equiv 1\bmod{\ell^k\Z_\ell[\omega]}$, it follows that $\pi_{\ell,k}(g')=\bar g$ and $\det g'=a\in\Z_\ell^\times$. Thus, $\bar g\in H_\ell(k)$. 

Next, we compute the cardinality of $H_\ell(k)$. Notice that the kernel of the reduction map $\phi_k:H_\ell(k+1)\longrightarrow H_\ell(k)$ consists of matrices of the form $I_2+\ell^k\gamma\bmod{\ell^{k+1}\Z_\ell[\omega]}$, with
$\gamma\in M_2(\Z_\ell[\omega]/\ell\Z_{\ell}[\omega])$ and $\det(I_2+\ell^k\gamma)\in (\Z/\ell^{k+1}\Z)^\times$. For $k\ge 1$, we have 
$$
\det(I_2+\ell^k\gamma)=1+\ell^k\tr\gamma+\ell^{2k}\det \gamma\equiv1+\ell^k\tr\gamma\bmod{\ell^{k+1}\Z_\ell[\omega]}.
$$
Write $\gamma$ explicitly
$$\gamma=\begin{pmatrix}
    a_0+a_1\omega & b_0+b_1\omega\\
    c_0+c_1\omega & d_0+d_1\omega
\end{pmatrix}\quad a_i,b_i,c_i,d_i\in\Z/\ell\Z.$$
Then $\det(I_2+\ell^k\gamma)\in (\Z/\ell^{k+1}\Z)^\times$ if and only if $\tr\gamma\in\Z/\ell\Z$,  i.e., $a_1+d_1=0$. Therefore, the coordinates $a_0,a_1,b_0, b_1, c_0,c_1,d_0$ may be chosen freely, each with $\ell$ possible values. Hence, $|\ker \phi_k|=\ell^7$. Since the groups $H_\ell(k+1)$ and $H_\ell(k)$ are both images of $H_\ell$, $\phi_k$ is surjective. It then follows that 
$$|H_\ell(k+1)|=|\ker \phi_k||H_\ell(k)|=\ell^7|H_\ell(k)|.$$
Put $h_\ell:=|H_\ell(1)|$. Iterating the preceding identity leads to
\begin{equation}\label{eqA2}
|H_\ell(k)|=h_\ell\ell^{7(k-1)}.
\end{equation}

For $k\geq1$, set $K_k:=\{g\in H_\ell:g\equiv I_2\bmod{\ell^k\Z_\ell[\omega]}\}=\ker\pi_{\ell,k}$.
In the $\ell$-adic topology on $H_\ell$, the subgroups $K_k$ form a fundamental system of neighborhoods of the identity.
For each $g\in U$, openness gives an integer $e_g\geq1$ such that
$gK_{e_g}\subseteq U$. By the compactness of $U$, we have
$$
U=\bigcup_{i=1}^s g_iK_{e_i}.
$$
Put $e:=\displaystyle\max_{1\leq i\leq s}e_i$.
Since $K_e\subseteq K_{e_i}$ for every $i$, each coset
$g_iK_{e_i}$ is a union of cosets of $K_e$. Hence \begin{equation}\label{eqA3}
U=\pi^{-1}_{\ell,e}(\pi_{\ell,e}(U)).
\end{equation}
For $k\geq e$, set
$$
S_k:=\{\bar g\in\pi_{\ell,k}(U):t(\bar g)\equiv0\bmod{\ell^k}\},\qquad N_k:=\#S_k,
$$
where $t(\bar g)$ denotes the reduction of $t(g)$ for any lift
$g\in H_\ell$ of $\bar g$ (this is well defined since congruent matrices have congruent traces, and comparison of the coefficients in the basis $(1,\omega)$ shows that their $\omega$ coefficients are congruent modulo $\ell^k$). By \eqref{eqA3}, we have $U=\pi^{-1}_{\ell,k}(\pi_{\ell,k}(U))$ for $k\ge e$. Consequently, 
$\{g\in U:t(g)\in\ell^k\Z_\ell\}=\pi^{-1}_{\ell,k}(S_k)$.
It then follows that 
\begin{equation}\label{eqA4}
\{g\in U:t(g)\in\ell^k\Z_\ell\}=\displaystyle\bigsqcup_{j=1}^{N_k}g_jK_k,
\end{equation}
for some $g_1,\dots,g_{N_k}\in H_\ell$.
As the Haar measure is invariant under translation, we have $\mu_\ell(g_jK_k)=\mu_\ell(K_k)$. Since $H_\ell$ is a disjoint union of $|H_\ell(k)|$ cosets of $K_k$ and $\mu_\ell(H_\ell)=1$, we see that $\mu_\ell(K_k)=|H_\ell(k)|^{-1}$.
 Thus \eqref{eqA2} and \eqref{eqA4} give
\begin{equation}\label{eqA5}
\ell^k\mu_\ell\left(\{g\in U:t(g)\in\ell^k\Z_\ell\}\right)
=\frac{\ell^7}{h_\ell}\,\ell^{-6k}N_k.
\end{equation}
It remains to prove the convergence of $\ell^{-6k}N_k$ as $k\to\infty$.
For $k\ge e$, from \eqref{eqA3} and \eqref{eqA1}, we have
\begin{equation}\label{eqA6}
\begin{aligned}
\bar g\in S_k
&\iff \bar g\bmod{\ell^e\Z_\ell[\omega]}\in \pi_{\ell,e}(U),\quad
\det\bar g\in (\Z/\ell^k\Z)^\times,\quad t(\bar g)\equiv0\bmod{\ell^k}; \\
&\iff
\begin{cases}
\displaystyle
\bar g=A+\omega B,\qquad
A=\begin{pmatrix}a&b\\ c&d\end{pmatrix},\qquad
B=\begin{pmatrix}x&y\\ z&-x\end{pmatrix},\\
(A+\omega B)\bmod{\ell^e\Z_\ell[\omega]}\in \pi_{\ell,e}(U),\\
(d-a)x-cy-bz-u(x^2+yz)\equiv0\bmod{\ell^k},
\end{cases}
\end{aligned}
\end{equation}
with $A,B\in M_2(\Z/\ell^k\Z)$. In fact, the last equivalence is obtained as follows; expanding $\det(A+\omega B)$ and using $\omega^2=u\omega-v$ gives
$$
\det(A+\omega B)=ad-bc+v(x^2+yz)
+\omega \bigl((d-a)x-cy-bz-u(x^2+yz)\bigr).
$$
On the other hand, the condition $A+\omega B\bmod{\ell^e\Z_\ell[\omega]}\in \pi_{\ell,e}(U)$ already ensures that $ad-bc+v(x^2+yz)\in(\Z/\ell^e\Z)^\times$, as $\pi_{\ell,e}(U)\subset H_\ell(e)$. Since $k\ge e$, it is therefore a unit modulo $\ell^k$ too. Thus, $\det(A+\omega B)\in(\Z/\ell^k\Z)^\times$ is equivalent to the last congruence in the brace of \eqref{eqA6}. If $B\neq0$, let $j$ be the largest integer such that every entry of
$B$ is divisible by $\ell^j$. Then $0\le j<k$, and we can write $B=\ell^j B_0$ for a unique $B_0\in M_2(\Z/\ell^{k-j}\Z)$ (we shall call such a matrix $B_0$ primitive, i.e., $B_0\not\equiv 0\bmod\ell$). We have
\begin{equation}\label{eqA7}
\begin{aligned}
\eqref{eqA6}\iff
\begin{cases}
\bar g=A+\omega\ell^j B_0,\qquad
A=\begin{pmatrix}a&b\\ c&d\end{pmatrix},\qquad
B_0=\begin{pmatrix}x_0&y_0\\ z_0&-x_0\end{pmatrix},\\
(A+\omega\ell^jB_0)\bmod\ell^e\Z_\ell[\omega]\in \pi_{\ell,e}(U),\\
\underbrace{(d-a)x_0-cy_0-bz_0-u\ell^j(x_0^2+y_0z_0)}_{Q_j(A,B_0)}
\equiv0\bmod{\ell^{k-j}}.
\end{cases}
\end{aligned}
\end{equation}
Let $N_{k,j}$ (resp. $N_{k,k}$) be the cardinality of the set of matrices $\bar{g}=A+\omega B$ in $S_k$ with $B=\ell^jB_0$ for $0\le j<k$ and $B_0$ primitive (resp. $B=0$). Then 
$$
N_k= \sum_{j=0}^{k}N_{k,j}.
$$
Let us start by counting $N_{k,j}$ for which $k\ge j+e$. Set
$$
\mathcal{C}_j:=\left\{(A,B)\in M_2(\Z/\ell^e\Z)^2\; :\; B\;\text{is primitive},\;\tr B=0,\;\begin{array}{l} A+\omega\ell^jB\in \pi_{\ell,e}(U),\\
Q_j(A,B)\equiv0\bmod{\ell^e}\end{array}\right\}.
$$
For $(A_0,B_{0,e})\in \mathcal{C}_j$, set
$$\mathcal B(B_{0,e}):=\left\{B_0\in M_2(\Z/\ell^{k-j}\Z):
\tr B_0=0,\quad B_0\bmod\ell^e=B_{0,e}
\right\},
$$
and for each $B_0\in\mathcal B(B_{0,e})$
$$
\A(A_0,B_0)
:=\left\{A\in M_2(\Z/\ell^k\Z):
A\bmod\ell^e=A_0,\quad Q_j(A,B_0)\equiv 0\bmod\ell^{k-j}
\right\}.
$$
We have
$$N_{k,j}=\sum_{(A_0,B_{0,e})\in\mathcal C_j}
  \sum_{B_0\in\mathcal B(B_{0,e})}
  \#\A(A_0,B_0).$$
We let $c_j$ denote the cardinality of the set $\mathcal{C}_j$ and fix $(A_0,B_{0,e})\in \mathcal{C}_j$. Note that each of the three entries of $B_{0,e}$ has $\ell^{k-j-e}$ lifts to $\Z/\ell^{k-j}\Z$. Therefore, $\# \mathcal{B}(B_{0,e})=\ell^{3(k-j-e)}$. Next, we fix $B_0\in \mathcal{B}(B_{0,e})$ and compute the cardinality of $\A(A_0,B_0)$. Let $A\in \A(A_0,B_0)$ and choose an entry of $A$ whose coefficient in $Q_j(A,B_0)$ is a unit. The other three entries may
be chosen arbitrarily modulo $\ell^k$ subject to the constraint $A\bmod\ell^e=A_0$. Hence, there are $\ell^{3(k-e)}$ choices for these coefficients. For every such choice, the congruence $Q_j(A,B_0)\equiv 0\bmod\ell^{k-j}$ gives exactly $\ell^j$ solutions for the remaining entry. Consequently, $\#\A(A_0,B_0)=\ell^{3(k-e)}\ell^j$. Thus, for $k\ge j+e$, we have
$$
N_{k,j}=c_j\ell^{3(k-j-e)}\ell^{3k+j-3e}=c_j\ell^{6k-6e-2j}.
$$ 
Equivalently
\begin{equation}\label{eqA8}
\ell^{-6k}N_{k,j}=c_j\ell^{-6e-2j}.
\end{equation}
We next bound $N_{k,j}$ for which $k<j+e$. For $0\le j<k$, if we drop the condition of membership from $\pi_{\ell,e}(U)$ in \eqref{eqA7}. Then, there are at most $\ell^{3(k-j)}$  choices for the primitive matrix $B_0$. As before, we fix $B_0$, and we choose an entry of $A$ whose coefficient in $Q_j(A,B_0)\equiv 0\bmod\ell^{k-j}$ is a unit. The remaining three coefficients have $\ell^{3k}$ choices. For each choice, the congruence \(Q_j(A,B_0)\equiv 0\pmod{\ell^{k-j}}\) determines the remaining entry uniquely modulo $\ell^{k-j}$, giving exactly $\ell^j$ lifts modulo $\ell^k$. Hence
$$N_{k,j}\le \ell^{3(k-j)}\ell^{3k}\ell^j,$$
or equivalently
\begin{equation}\label{eqA9}
\ell^{-6k}N_{k,j}\le \ell^{-2j}.    
\end{equation}
Finally, for $B=0$, only the four entries of $A$ remain, giving $N_{k,k}\le \ell^{4k}$, and hence $\ell^{-6k}N_{k,k}\le \ell^{-2k}$. Combining \eqref{eqA8} and \eqref{eqA9}, we obtain
$$\ell^{-6k}N_k=\sum_{j=0}^{k-e}c_j\ell^{-6e-2j}+R_k,\quad\text{where}\quad R_k=\sum_{j=k-e+1}^k \ell^{-6k}N_{k,j}.$$
Since 
$$
0\le R_k\le \sum_{j=k-e+1}^{k-1} \ell^{-2j}+\ell^{-2k}\ll_\ell \ell^{-2(k-e+1)},
$$
we deduce that $R_k\to 0$ as $k\to\infty$. Letting $k\to\infty$ in \eqref{eqA5}, we get
$$L_\ell(U)=\lim_{k\to\infty}\ell^{k}\mu_\ell\left(\{g\in U:t(g)\in\ell^k\Z_\ell\}\right)
=\frac{\ell^{7-6e}}{h_\ell}\,\sum_{j=0}^\infty c_j\ell^{-2j}$$
which converges as $c_j\ell^{-6e-2j}\le\ell^{-2j}$. Moreover, for $j\ge e$ we have $c_j=c_e$, thus
$$L_\ell(U)=\frac{\ell^{7-6e}}{h_\ell}\left(
\sum_{j=0}^{e-1}c_j\ell^{-2j}
+\frac{c_e\ell^{-2e}}{1-\ell^{-2}}
\right).$$

Suppose now that $I_2\in U$; then for every $B=\begin{pmatrix}
    x & y \\
    z & -x
\end{pmatrix}\in M_2(\Z/\ell^e\Z)$ primitive, we have $(I_2,B)\in \mathcal{C}_e$. As $Q_e(I_2,B)=-u\ell^e(x^2+yz)\equiv0\bmod\ell^e$, we see that $c_e\ge \ell^{3e}-\ell^{3(e-1)}>0$. Consequently, $L_\ell(U)>0$.
\end{proof}
\begin{proof}[Proof of \thref{propfinite-placeconstant}]
By our choice of $m_f$, we have $\ker\pi_{m_f}\subset\Gamma$. Consequently 
\begin{equation}\label{eqA10}
\Gamma=\Gamma_{m_f}\times\prod_{\ell\nmid m_f}H_\ell,
\end{equation}
with $\Gamma_{m_f}:={\rm pr}_{m_f}(\Gamma)$, where ${\rm pr}_{m_f} : \Gamma\longrightarrow \displaystyle H_{m_f}:=\prod_{\ell\mid m_f}H_\ell$ is the natural projection. We equip $H_{m_f}$ with the product measure
$\mu_{m_f}:=\displaystyle\prod_{\ell\mid m_f}\mu_\ell$. 

On the other hand, for $g\in\GL_2(\mathcal O_{f,\ell})$, since $\overline\omega=u-\omega$, we have $\tr g-\overline{\tr g}=(2\omega-u)t(g).$ 
Set $r_\ell:=\min\{v_\ell(2),v_\ell(u)\}$. Therefore, for $k>r_\ell$,
$$
\tr g-\overline{\tr g}\in\ell^k\mathcal O_{f,\ell}
\quad\Longleftrightarrow\quad
t(g)\in\ell^{k-r_\ell}\Z_\ell.
$$
It then follows from \thref{lemlocaldensity} that for any compact open $U\subset H_\ell$, we have
\begin{equation}\label{eqA11}
\Lambda_\ell(U):=\lim_{k\to\infty}\ell^k
\mu_\ell\left\{g\in U:
\tr g-\overline{\tr g}\in\ell^k\Z_{\ell}[\omega]\right\}
=\ell^{r_\ell}L_\ell(U).
\end{equation}
For $\ell\mid m_f$, put $e_\ell=v_\ell(m_f)$ and $K_{e_\ell}=\ker\pi_{\ell,e_\ell}$. By our choice of $m_f$ we have $\ker\pi_{m_f}\subset \Gamma$. Therefore
$$
K:=\prod_{\ell\mid m_f}K_{e_\ell}\subseteq \Gamma_{m_f}
$$
 Furthermore, $\Gamma_{m_f}/K$ is a finite group, as it is contained in $\displaystyle H_{m_f}/K\cong\prod_{\ell\mid m_f}H_\ell(e_\ell)$. Let $\gamma_1=1,\gamma_2,\dots,\gamma_h\in\Gamma_{m_f}$ be a system of representatives for the cosets of $K$ in $\Gamma_{m_f}$, then we have 
$$
\Gamma_{m_f}=\bigsqcup_{a=1}^{h}\gamma_a K.
$$
Next, we write $\gamma_{a}=(\gamma_{a,\ell})_{\ell\mid m_f}$ and put $U_{a,\ell}:=\gamma_{a,\ell}K_{e_\ell}$. Note that each $U_{a,\ell}$ is compact and open and we have 
$$
\Gamma_{m_f}
=\bigsqcup_{a=1}^h\prod_{\ell\mid m_f}U_{a,\ell}.
$$
On the other hand, for $k\ge1$ set
$$E_{\ell,k}=\{g\in H_{\ell}\; :\; \tr g-\overline{\tr g}\in\ell^k\Z_\ell[\omega]\}.$$
Since the uniform Haar measure on $\Gamma_{m_f}$ pushes forward to uniform measure on its finite reduction, we have
\begin{align}\label{eqFmb}
F_{m_b}&=\frac{\prod_{\ell\mid m_f}\ell^{v_\ell(m_b)}}
     {\mu_{m_f}(\Gamma_{m_f})}
\mu_{m_f}\!\left(
\Gamma_{m_f}\cap\prod_{\ell\mid m_f}E_{\ell,v_\ell(m_b)}
\right)\nonumber\\
&=\frac{1}{\mu_{m_f}(\Gamma_{m_f})}
\sum_{a=1}^{h}
\prod_{\ell\mid m_f}
\ell^{v_\ell(m_b)}
\mu_\ell\!\left(U_{a,\ell}\cap E_{\ell,v_\ell(m_b)}\right).
\end{align}
It then follows from \eqref{eqA11} that
\begin{equation}\label{eqA12}
\widehat F_{m_f}
:=\lim_{\substack{k_\ell\to\infty\\\ell\mid m_f}}
F_{\prod_{\ell\mid m_f}\ell^{k_\ell}}
=\frac{1}{\mu_{m_f}(\Gamma_{m_f})}
\sum_{a=1}^h\prod_{\ell\mid m_f}\Lambda_\ell(U_{a,\ell})
<\infty.
\end{equation}
As $I_2\in U_{1,\ell}$ for all $\ell\mid m_f$ by \thref{lemlocaldensity} we have $\Lambda_\ell(U_{1,\ell})>0$ for all $\ell\mid m_f$, and hence $\widehat{F}_{m_f}>0$.

Let $m$ be an arbitrary positive integer, write $m=m_bm_g$, where every prime divisor of $m_b$ divides $m_f$ and $(m_g,m_f)=1$. The Chinese remainder theorem, together with \eqref{eqA10} yields
\begin{equation}\label{eqA13}
F_m=F_{m_b}\prod_{\ell\mid m_g}F_{\ell^{v_\ell(m_g)}}.
\end{equation}
A prime $\ell\nmid m_f$ is odd, unramified in $K_f$ and has large image. It then follows  from \cite[Propositions 5.5--5.6--5.7--5.8]{van2017} that for $k\ge 1$ we have
$$
F_{\ell^k}=
\begin{cases}
\displaystyle
\frac{\ell^3(\ell^2+\ell+1-\ell^{-2k})}
     {(\ell+1)(\ell^4-1)},
&\ell\text{ inert in }K_f,\\
\displaystyle
\frac{\ell(\ell^4+\ell^3-\ell^2-2\ell-\ell^{2-2k})}
     {(\ell+1)(\ell^2-1)^2},
&\ell\text{ split in }K_f.
\end{cases}
$$
It follows that $F_{\ell^k}=1+O\left(\frac{1}{\ell^2}\right)$ uniformly in $k\ge 1$. By \cite[Lemma 5.9]{van2017} we have $F_{\ell^k}\to\widetilde{F}_\ell$ as $k\to\infty$. Moreover, since $\widetilde{F}_\ell=1+O\left(\frac{1}{\ell^2}\right)$,  the product $\displaystyle\prod_{\ell\nmid m_f}\widetilde{F}_\ell$ is finite and strictly positive. By telescoping, we see
$$
\left|\prod_{\ell\nmid m_f}\widetilde F_\ell-\prod_{\ell \mid m_g}F_{\ell^{v_\ell(m_g)}}\right|\ll\sum_{\ell\nmid m_f}\left|\widetilde F_\ell-F_{\ell^{v_\ell(m_g)}}\right|.
$$
For every $Y\ge 2$, split the sum at $Y$
$$
\left|\prod_{\ell\nmid m_f}\widetilde F_\ell-\prod_{\ell \mid m_g}F_{\ell^{v_\ell(m_g)}}\right|\ll \sum_{\substack{\ell\nmid m_f\\ \ell\le Y}}\left|\widetilde F_\ell-F_{\ell^{v_\ell(m_g)}}\right|+\sum_{\substack{\ell\nmid m_f\\ \ell> Y}}\left|\widetilde F_\ell-F_{\ell^{v_\ell(m_g)}}\right|.
$$
For each $\ell\le Y$ outside $m_f$, $v_\ell(m_g)\to\infty$ along divisibility, while the tail is $O\left(\displaystyle \sum_{\ell>Y}\frac{1}{\ell^{2}}\right)$, uniformly in $m$. It then follows that $F_{m_b}\to \widehat{F}_{m_f}$ and $F_{m_g}\to\displaystyle\prod_{\ell\nmid m_f}\widetilde F$ along divisibility leading to the stated Euler product formula. Finally, since $\widehat{F}_{m_f}$ and the product are both finite and strictly positive, we have $\widehat F_f\in(0,\infty)$.
\end{proof}
\begin{rem}\thlabel{apprem}
\begin{enumerate}
\item Let $m$ be an arbitrary integer $m$ write $m=m_b m_g$. By \eqref{eqA13}, we have
$$
F_m=F_{m_b}\prod_{\ell\mid m_g}F_{\ell^{v_\ell(m_g)}}.
$$
By \eqref{eqA11}, for each $a$ and $\ell\mid m_f$, the sequence $\bigl(\ell^k\mu_\ell(U_{a,\ell}\cap E_{\ell,k})\bigr)_{k\ge0}$ is bounded. Since \eqref{eqFmb} is a finite sum of products of these quantities, $F_{m_b}=O_f(1)$ uniformly in $m_b$. On the other hand, there exists $C>0$ such that $F_{\ell^k}\le 1+C/\ell^2$ uniformly in $\ell\nmid m_f$ and $k\ge1$. Thus
$$
\prod_{\ell\mid  m_g}F_{\ell^{v_\ell(m_g)}}
\le \prod_{\ell\nmid m_f}\left(1+\frac C{\ell^2}\right)<\infty.
$$
In consequence, we get $F_m=O_f(1)$.
\item The Chinese remainder theorem gives an injection
$$
\G_m\hookrightarrow\prod_{\ell^k\parallel m}H_\ell(k).
$$ 
By \eqref{eqA2}, we have $|H_\ell(k)|=\ell^{7k}\frac{h_\ell}{\ell^7}$. Since the determinant map from $H_\ell(1)$  onto $\mathbb{F}_\ell^\times$ is surjective, one has $h_\ell=(\ell-1)|\SL_2(\Z_\ell[\omega]/\ell\Z_\ell[\omega])|$. For an unramified prime $\ell$ in $K_f$, we have
$$
\Z_\ell[\omega]/\ell\Z_\ell[\omega]\cong
\begin{cases}
\mathbb{F}_{\ell^2} & \ell\text{ inert in }K_f,\\
\mathbb{F}_\ell\times\mathbb{F}_\ell & \ell\text{ split in }K_f.
\end{cases}
$$
Thus, only the finitely many ramified primes can contribute a factor greater than $1$ to the normalized product. Therefore
$$
|\G_m|
\le m^7\prod_{\ell\mid m}\frac{h_\ell}{\ell^7}
\le m^7\prod_{\ell\ \mathrm{ramified\ in}\ K_f}
          \max\!\left\{1,\frac{h_\ell}{\ell^7}\right\}
\ll_f m^7.
$$
\end{enumerate}  
\end{rem}
\printbibliography[heading=bibintoc,title=References]

@article{amri2023,
  author = {Amri, Mohammed Amin},
  title = {{Chebotarev--Sato--Tate} distribution for abelian surfaces potentially of {$\mathrm{GL}_2$}-type},
  journal = {European Journal of Mathematics},
  volume = {9},
  number = {4},
  pages = {86},
  year = {2023},
  doi = {10.1007/s40879-023-00682-5}
}

@article{amri2024,
  author    = {Mohammed Amin Amri},
  title     = {On the Lang--Trotter conjecture for a class of non-generic abelian surfaces},
  journal   = {The Ramanujan Journal},
  volume    = {64},
  pages     = {891--908},
  year      = {2024},
  doi       = {10.1007/s11139-024-00884-9},
}

@article{BanaszakKedlaya2015,
  author = {Banaszak, Grzegorz and Kedlaya, Kiran S.},
  title = {An algebraic {Sato--Tate} group and {Sato--Tate} conjecture},
  journal = {Indiana University Mathematics Journal},
  volume = {64},
  number = {1},
  pages = {245--274},
  year = {2015},
  doi = {10.1512/iumj.2015.64.5438}
}

@incollection{bucur2016,
  author = {Bucur, Alina and Kedlaya, Kiran S.},
  title = {An application of the effective {Sato--Tate} conjecture},
  booktitle = {Frobenius Distributions: {Lang--Trotter} and {Sato--Tate} Conjectures},
  series = {Contemporary Mathematics},
  volume = {663},
  publisher = {American Mathematical Society},
  address = {Providence, RI},
  pages = {45--56},
  year = {2016},
  doi = {10.1090/conm/663/13349}
}

@article{bucur2024,
  author = {Bucur, Alina and Fit{\'e}, Francesc and Kedlaya, Kiran S.},
  title = {Effective {Sato--Tate} conjecture for abelian varieties and applications},
  journal = {Journal of the European Mathematical Society},
  volume = {26},
  number = {5},
  pages = {1713--1746},
  year = {2024},
  doi = {10.4171/JEMS/1443}
}

@book{Bump2013,
  author = {Bump, Daniel},
  title = {Lie Groups},
  edition = {Second},
  series = {Graduate Texts in Mathematics},
  volume = {225},
  publisher = {Springer},
  address = {New York},
  year = {2013},
  doi = {10.1007/978-1-4614-8024-2}
}

@article{cantoral2022,
  author = {Cantoral-Farf{\'a}n, Victoria and Commelin, Johan},
  title = {The {Mumford--Tate} conjecture implies the algebraic {Sato--Tate} conjecture of {Banaszak} and {Kedlaya}},
  journal = {Indiana University Mathematics Journal},
  volume = {71},
  number = {6},
  pages = {2595--2603},
  year = {2022},
  doi = {10.1512/iumj.2022.71.9309}
}

@article{colzani2011,
  author = {Colzani, Leonardo and Gigante, Giacomo and Travaglini, Giancarlo},
  title = {Trigonometric approximation and a general form of the {Erd{\H{o}}s Tur{\'a}n} inequality},
  journal = {Transactions of the American Mathematical Society},
  volume = {363},
  number = {2},
  pages = {1101--1123},
  year = {2011},
  doi = {10.1090/S0002-9947-2010-05287-0}
}

@article{fite2012,
  author = {Fit{\'e}, Francesc and Kedlaya, Kiran S. and Rotger, V{\'i}ctor and Sutherland, Andrew V.},
  title = {{Sato--Tate} distributions and {Galois} endomorphism modules in genus 2},
  journal = {Compositio Mathematica},
  volume = {148},
  number = {5},
  pages = {1390--1442},
  year = {2012},
  doi = {10.1112/S0010437X12000279}
}

@article{Hormander1968,
  author = {H{\"o}rmander, Lars},
  title = {The spectral function of an elliptic operator},
  journal = {Acta Mathematica},
  volume = {121},
  pages = {193--218},
  year = {1968},
  doi = {10.1007/BF02391913}
}

@book{knapp2002,
  title     = {Lie Groups Beyond an Introduction},
  author    = {Knapp, Anthony W.},
  edition   = {2nd},
  series    = {Progress in Mathematics},
  volume    = {140},
  year      = {2002},
  publisher = {Birkh{\"a}user},
  address   = {Boston},
  isbn      = {978-0-8176-4259-4}
}

@incollection{lagarias1977,
  author = {Lagarias, Jeffrey C. and Odlyzko, Andrew M.},
  title = {Effective versions of the {Chebotarev} density theorem},
  booktitle = {Algebraic Number Fields: {$L$}-Functions and {Galois} Properties},
  editor = {Fr{\"o}hlich, Albrecht},
  publisher = {Academic Press},
  address = {London},
  pages = {409--464},
  year = {1977}
}

@article{mazur2008,
  author = {Mazur, Barry},
  title = {Finding meaning in error terms},
  journal = {Bulletin of the American Mathematical Society},
  volume = {45},
  number = {2},
  pages = {185--228},
  year = {2008}
}

@article{Murty1985,
  author = {Murty, V. Kumar},
  title = {Explicit formulae and the {Lang--Trotter} conjecture},
  journal = {Rocky Mountain Journal of Mathematics},
  volume = {15},
  number = {2},
  pages = {535--551},
  year = {1985},
  doi = {10.1216/RMJ-1985-15-2-535}
}

@incollection{murty1999,
  author = {Murty, V. Kumar},
  title = {Frobenius distributions and {Galois} representations},
  booktitle = {Automorphic Forms, Automorphic Representations, and Arithmetic},
  series = {Proceedings of Symposia in Pure Mathematics},
  volume = {66, Part 1},
  publisher = {American Mathematical Society},
  address = {Providence, RI},
  pages = {193--211},
  year = {1999}
}

@article{serre1970,
  author  = {Serre, Jean-Pierre},
  title   = {Facteurs locaux des fonctions z{\^e}ta des
             vari{\'e}t{\'e}s alg{\'e}briques
             (d{\'e}finitions et conjectures)},
  journal = {S{\'e}minaire Delange--Pisot--Poitou.
             Th{\'e}orie des nombres},
  volume  = {11},
  number  = {2},
  year    = {1969--1970},
  pages   = {1--15},
  note    = {Expos{\'e} no. 19},
  url     = {https://numdam.org/item/SDPP_1969-1970__11_2_A4_0/}
}

@book{serre2016,
  author = {Serre, Jean-Pierre},
  title = {Lectures on {$N_X(p)$}},
  series = {Research Notes in Mathematics},
  volume = {11},
  publisher = {CRC Press},
  year = {2016}
}

@article{van2017,
  author = {Van Hirtum, Jasper},
  title = {On the distribution of {Frobenius} of weight 2 eigenforms with quadratic coefficient field},
  journal = {Experimental Mathematics},
  volume = {26},
  number = {2},
  pages = {165--188},
  year = {2017},
  doi = {10.1080/10586458.2016.1142912}
}

@article{wong2019,
  author = {Wong, Peng-Jie},
  title = {On the {Chebotarev--Sato--Tate} phenomenon},
  journal = {Journal of Number Theory},
  volume = {196},
  pages = {272--290},
  year = {2019},
  doi = {10.1016/j.jnt.2018.09.010}
}

@book{diamond2005,
  title={A first course in modular forms},
  author={Diamond, Fred and Shurman, Jerry},
  volume={228},
  year={2005},
  publisher={Springer}
}

@book{lang2006,
  title={Frobenius distributions in $\GL_2$-extensions: Distribution of Frobenius automorphisms in $GL_2$-extensions of the rational numbers},
  author={Lang, Serge and Trotter, Hale},
  year={2006},
  publisher={Springer}
}

@inproceedings{ribet1979,
  title={Endomorphism algebras of abelian varieties attached to newforms of weight 2},
  author={Ribet, KA},
  booktitle={Seminar on Number Theory, Paris},
  volume={Progress in Mathematics 12, Birkhäuser},
  pages={263--276},
  year={1979}
}

@incollection{DM1982,
  AUTHOR    = {Deligne, Pierre and Milne, James S.},
  TITLE     = {Tannakian categories},
  BOOKTITLE = {Hodge cycles, motives, and {S}himura varieties},
  SERIES    = {Lecture Notes in Mathematics},
  VOLUME    = {900},
  PUBLISHER = {Springer-Verlag},
  ADDRESS   = {Berlin-New York},
  YEAR      = {1982},
  PAGES     = {101--228}
}

@article{Kumar2026,
  title     = {Effective Joint Sato--Tate Distribution and Sign Change of Symmetric Power Coefficients},
  author    = {Kumar, Arvind and Kumari, Moni and Mishra, Prabhat Kumar},
  journal   = {arXiv preprint arXiv:2604.17532},
  year      = {2026},
  note      = {v2, published July 2026},
  url       = {https://arxiv.org/abs/2604.17532}
}

@article{loeffler2017,
  author  = {Loeffler, David},
  title   = {Images of adelic {G}alois representations for modular forms},
  journal = {Glasgow Mathematical Journal},
  volume  = {59},
  number  = {1},
  pages   = {11--25},
  year    = {2017},
  publisher={Cambridge University Press}
}

@article{Cojocaru2016,
  author    = {Alina Carmen Cojocaru and Rachel Davis and Alice Silverberg and Katherine E. Stange},
  title     = {Arithmetic properties of the Frobenius traces defined by a rational abelian variety (with two appendices by J-P. Serre)},
  journal   = {International Mathematics Research Notices},
  volume    = {2016},
  number    = {1},
  pages     = {93--109},
  year      = {2016},
  doi       = {10.1093/imrn/rnw058},
  url       = {https://doi.org}
}

@article{bm1988,
  title     = {Semianalytic and subanalytic sets},
  author    = {Edward Bierstone and Pierre D. Milman},
  journal   = {Publications Mathématiques de l'IHÉS},
  year      = {1988},
  volume    = {67},
  pages     = {5--42},
  doi       = {10.1007/BF02699126},
  url       = {https://www.numdam.org/item/PMIHES_1988__67__5_0/}
}

@book{springer1998,
  author    = {Springer, T. A.},
  title     = {Linear Algebraic Groups},
  edition   = {2nd},
  series    = {Progress in Mathematics},
  volume    = {9},
  publisher = {Birkh{\"a}user},
  address   = {Boston, MA},
  year      = {1998},
  isbn      = {978-0-8176-4021-7},
  doi       = {10.1007/978-1-4612-0941-6},
  url       = {https://doi.org/10.1007/978-1-4612-0941-6}
}
\Addresses
\end{document}